\documentclass[11pt,a4paper,final]{article}
\pdfoutput=1
\usepackage{amssymb,latexsym}

\input{epsf.sty}
\usepackage{epsfig}
\usepackage{graphicx,color}
\graphicspath{{./Figures/}}
\usepackage{epstopdf}

\usepackage{setspace}

\usepackage{amsmath}

\usepackage{mathrsfs,amsfonts}

\usepackage{amsthm,amsxtra}

\usepackage[final]{showkeys}

\usepackage{stmaryrd}

\usepackage{algorithm}
\usepackage{algorithmicx}
\usepackage{algpseudocode}
\usepackage{mathtools}
\newif\ifPDF
\ifx\pdfoutput\undefined
\PDFfalse
\else
\ifnum\pdfoutput > 0
\PDFtrue
\else
\PDFfalse
\fi
\fi

\ifPDF
\usepackage{pdftricks}
\begin{psinputs}
	\usepackage{pstricks}
	\usepackage{pstcol}
	\usepackage{pst-plot}
	\usepackage{pst-tree}
	\usepackage{pst-eps}
	\usepackage{multido}
	\usepackage{pst-node}
	\usepackage{pst-eps}
\end{psinputs}
\else
\usepackage{pstricks}
\fi

\ifPDF
\usepackage[debug,pdftex,colorlinks=true, 
linkcolor=blue, bookmarksopen=false,
plainpages=false,pdfpagelabels]{hyperref}
\else
\usepackage[dvips]{hyperref}
\fi

\usepackage[openbib]{currvita}

\usepackage{fancyhdr}

\newtheorem{theorem}{Theorem}[section]
\newtheorem{lemma}[theorem]{Lemma}
\newtheorem{definition}[theorem]{Definition}

\newtheorem{proposition}[theorem]{Proposition} 
\newtheorem{remark}[theorem]{Remark}

\newcommand{\diag}{{\rm diag}}

\newcommand{\eps}{\varepsilon}
\newcommand{\GL}{G_{\Lambda^*_0}}

\newcommand{\tG}{\tilde {G}}
\newcommand{\SH}{\cS_{H_0,\eps}}
\newcommand{\SL}{\cS_{\Lambda^*_0,\eps}}

\newcommand{\tS}{\tilde{\cS}_{\eps}}
\newcommand{\KH}{\cK_{H_0,\eps}}
\newcommand{\KL}{\cK_{\Lambda^*_0,\eps}}
\newcommand{\tK}{\tilde{\cK}_{\eps}}

\newcommand{\sfc}{\mathsf c}

\newcommand{\bbR}{\mathbb R} 
\newcommand{\bbZ}{\mathbb Z} 

\newcommand{\bzero}{{\mathbf 0}}

\newcommand{\bkappa}{{\boldsymbol\kappa}}

\newcommand{\ba}{\mathbf a} \newcommand{\bb}{\mathbf b}
\newcommand{\bc}{\mathbf c} 
\newcommand{\be}{\mathbf e}

\newcommand{\bq}{\mathbf q} \newcommand{\br}{\mathbf r}

 \newcommand{\bx}{\mathbf x} 
\newcommand{\by}{\mathbf y}

\newcommand{\cA}{\mathcal A} \newcommand{\cB}{\mathcal B}
 \newcommand{\cD}{\mathcal D} 
\newcommand{\cE}{\mathcal E} 
 
\newcommand{\cI}{\mathcal I} 
\newcommand{\cK}{\mathcal K}

\newcommand{\cS}{\mathcal S} \newcommand{\cT}{\mathcal T}

\newcommand{\abs}[1]{\left\vert#1\right\vert}
\newcommand{\norm}[1]{\|#1\|}

\newenvironment{keywords}
{\noindent{\bf Key words}\small}{\par\vspace{1ex}}
\newenvironment{AMS}
{\noindent{\bf AMS subject classifications}\small}{\par}

\makeatletter
\newcommand{\chapterauthor}[1]{%
	{\parindent0pt\vspace*{-25pt}%
		\linespread{1.1}\large\scshape#1%
		\par\nobreak\vspace*{35pt}}
	\@afterheading%
}
\makeatother
\title{Dirac points for the honeycomb lattice with impenetrable obstacles}
\author{
Wei Li \thanks{Department of Mathematical Sciences, DePaul University, Chicago, IL 60614. Email: wei.li@depaul.edu.} \; 
Junshan Lin \thanks{\footnotesize Department of Mathematics and Statistics, Auburn University, Auburn, AL 36849. Email: jzl0097@
auburn.edu. Junshan Lin was partially supported by the NSF grant DMS-2011148.}
 \; and Hai Zhang\thanks{\footnotesize 
Department of Mathematics, 
 HKUST,  Clear Water Bay, Kowloon, Hong Kong SAR, China. Email: haizhang@ust.hk. Hai Zhang is partially supported by Hong Kong RGC grant GRF 16305419 and GRF 16304621.}}
 \date{}
\begin{document}
\maketitle

\begin{abstract}
This work is concerned with the Dirac points for the honeycomb lattice with impenetrable obstacles arranged periodically in a homogeneous medium.  We consider both the Dirichlet and Neumann eigenvalue problems  and prove the existence of Dirac points for both eigenvalue problems at crossing of the lower band surfaces as well as higher band surfaces. Furthermore, we perform quantitative analysis for the eigenvalues and the slopes of two conical dispersion surfaces near each Dirac point based on a combination of the layer potential technique and asymptotic analysis. It is shown that the eigenvalues are in the neighborhood of the singular frequencies associated with the Green's function for the honeycomb lattice, and the slopes of the dispersion surfaces are reciprocal to the eigenvalues.
\end{abstract}


\section{Introduction}
Inspired from the discovery of the quantum Hall effects and topological insulators in condensed matter, there has been increasing interest in the exploration of topological photonic/phononic materials recently to manipulate photons/phonons the same way as solids modulating electrons \cite{HR-08, Khanikaev-12, Lu-14, Ozawa-19, Rechtsman-13, Yang-15}. These topological materials allow for the propagation of robust waveguide modes (or so-called edge modes) along the material interfaces without backscattering and even at the presence of large disorder, which could provide revolutionary applications for the design of novel optical/acoustic devices.

Typically the topological photonic/phononic materials are periodic band-gap media with the topological phases associated with the band structures of the underlying differential operators. The band gap is opened at certain special conical vertex of the band structure by breaking the time-reversal symmetry or the space-inversion symmetry of the periodic media  \cite{Fefferman-Thorp-Weinsein-16, Fefferman-Lee-Thorp-Weinsein-17, Lee-Thorp-Weinstein-Zhu-19, LZ21-2,  ma-shvets, makwana-19, wang-08, wu-hu-15}. Such vertices in the dispersion relation are called Dirac points, which emerge from the touching of two bands of the spectrum in a linear conical fashion, and their investigations play an important role in the design of novel topological materials. 

The mathematical analysis of Dirac point dates back to the study of the tight-binding approximation model for graphene in \cite{slonczewski-weiss-58, wallace-47} by Wallace, Slonczewski and Weiss, and more recently for a more generalized quantum graph model with potential on the edges of the honeycomb lattice in \cite{Kuchment-Post-07} by Kuchment and Post. Dirac points for the Schr\"{o}dinger equation model of graphene were considered in \cite{grushin-09} by Grushin over the honeycomb lattice with a weak potential and later thoroughly studied in \cite{Fefferman-Weinsein-12} by Fefferman and Weinstein for potentials that are not necessarily weak; See also \cite{berkokaiko-comech} for an alternative proof of the existence and stability of Dirac points for the Schr\"{o}dinger operator. These results are then generalized to a broad class of  elliptic operator defined over the honeycomb lattice, including the configurations with point scatterers, high-contrast medium, resonant bubbles, etc \cite{ammari-20-4, Cassier-Weinstein-21, Fefferman-Thorp-Weinsein-18, Lee-16, Lee-Thorp-Weinstein-Zhu-19}. We also refer the readers to \cite{ochiai-09, torrent-12, wang-08} for the numerical and experimental investigation of Dirac points in other acoustic and electromagnetic media.

In this paper we study the Dirac points for the honeycomb lattice with impenetrable obstacles embedded in a homogeneous medium. The setup arises naturally in photonic/phononic materials when the inhomogeneities are sound soft/hard in acoustic media or perfect electric/magnetic conducting in electromagnetic media. More precisely, we consider the honeycomb lattice in $\bbR^2$ given by $$ \Lambda := \bbZ \be_1 \oplus  \bbZ \be_2 := \{ \ell_1 \be_1 + \ell_2 \be_2 :  \ell_1, \ell_2 \in \bbZ \}, $$
where the lattice vectors 
$\be_1= a\left(\frac{\sqrt{3}}{2},\frac{1}{2}\right)^T$, $\be_2= a\left(\frac{\sqrt{3}}{2},-\frac{1}{2}\right)^T$, 
and the lattice constant is $a$. Let $Y:= \{ t_1 \be_1 + t_2 \be_2 \, | \, 0 \le t_1, t_2 \le 1 \}$ be the fundamental cell of the lattice, which contains a circular shape impenetrable obstacle $D_\eps$ with radius $\eps$ centered at $x_c=\frac{1}{2}(\be_1+\be_2)$ (see Figure \ref{fig:honeycomb_structure}, left). $Y_\eps:= Y\backslash D_\eps$ denotes the domain exterior to the obstacle in the fundamental cell. The reciprocal lattice vectors $\bkappa_1$ and $\bkappa_2$ are
 $\bkappa_1= \frac{2\pi}{a}\left(\frac{\sqrt{3}}{3},1\right)^T$ and $\bkappa_2=  \frac{2\pi}{a}\left(\frac{\sqrt{3}}{3},-1\right)^T$,  which satisfy  $ \be_i \cdot \bkappa_j = 2\pi \delta_{ij}$ for $i, j = 1, 2$. The reciprocal lattice is given by
$$
\Lambda^*=\bbZ \bkappa_1 \oplus  \bbZ \bkappa_2 : = \left\{\ell_1\bkappa_1+\ell_2 \bkappa_2:   \ell_1, \ell_2 \in \bbZ \right\}.
$$
The hexagon shape of the fundamental cell in $\Lambda^*$, or the Brillouin zone, is denoted by $\cB$ and shown in Figure \ref{fig:honeycomb_structure} (right).

\begin{figure}[!htbp]
    \centering
    \vspace{-10pt}
    \includegraphics[width=8.5cm]{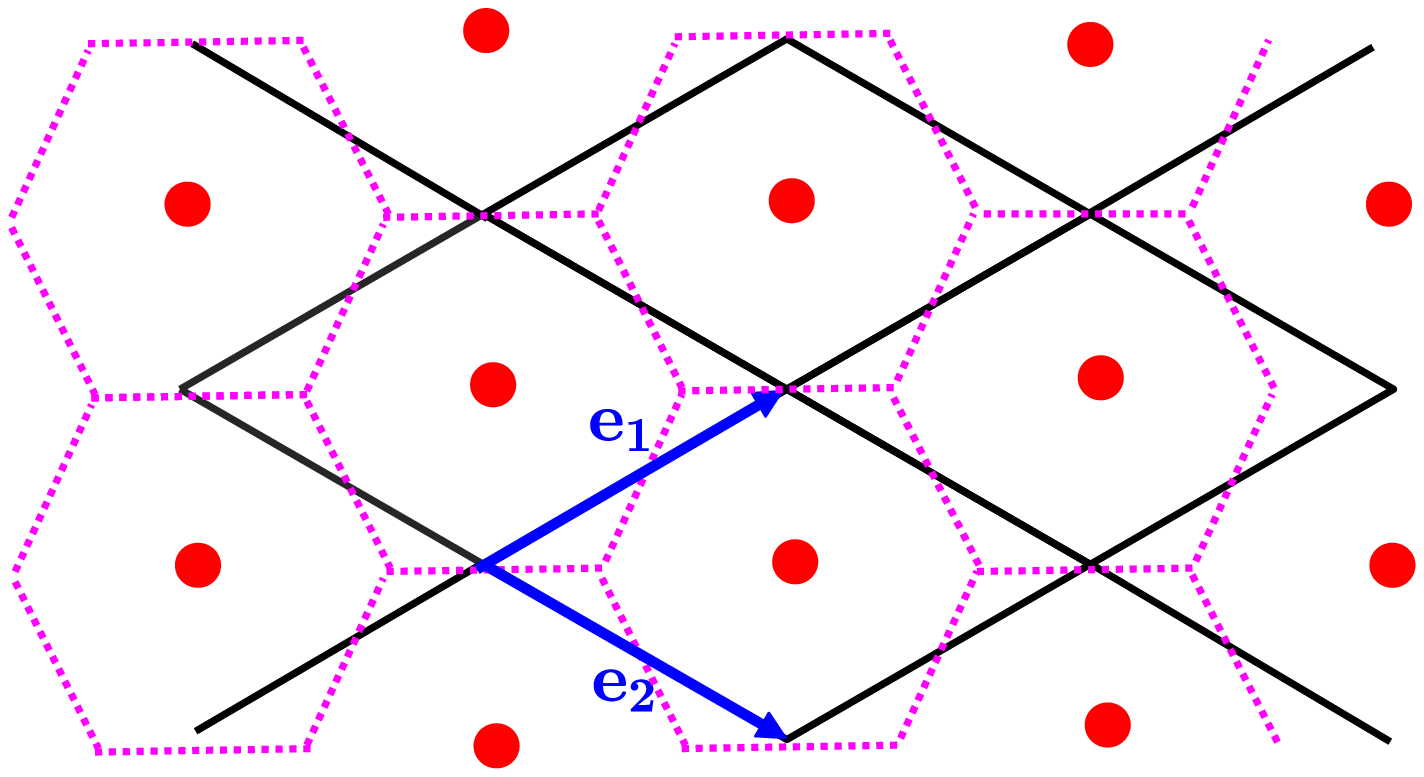} 
    \includegraphics[width=7cm]{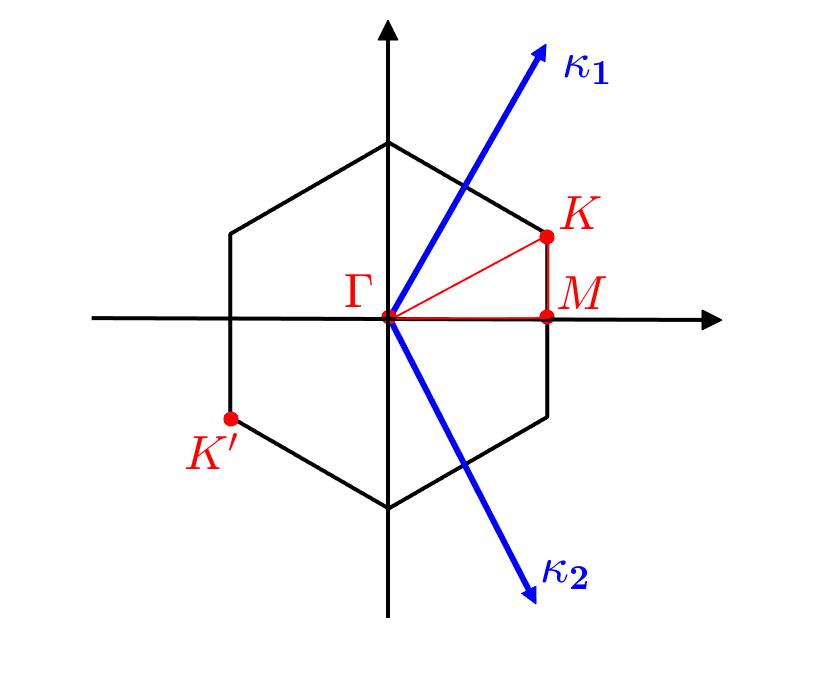}
    \vspace{-20pt}
    \caption{Left: Honeycomb lattice with impenetrable obstacles located in the cell centers. The lattice vectors $\be_1= a\left(\frac{\sqrt{3}}{2},\frac{1}{2}\right)^T$ and $\be_2= a\left(\frac{\sqrt{3}}{2},-\frac{1}{2}\right)^T$. The lattice constant is $a$ and the size of each obstacle is $\eps$. Right: Brillouin zone generated by the reciprocal lattice vectors $\bkappa_1= \frac{2\pi}{a}\left(\frac{\sqrt{3}}{3},1\right)^T$ and $\bkappa_2=  \frac{2\pi}{a} \left(\frac{\sqrt{3}}{3},-1\right)^T$. The high symmetry vertices $K=\frac{2\pi}{a}(\frac{1}{\sqrt3},\frac{1}{3})^T$ and $K'=-K$, and the vertices
    $\Gamma = (0, 0)^T$, $M=\frac{2\pi}{a}(\frac{1}{\sqrt3},0)^T$.}
    \label{fig:honeycomb_structure}
\end{figure}

For each Bloch wave vector $\bkappa\in \cB$, we consider the following eigenvalue problem with the frequency $\omega \in \mathbb R$:
\begin{equation}\label{eq:eigen_prob}
\begin{aligned}
&\Delta  \psi(\bx) +\omega^2 \psi(\bx)=0,  \quad &\bx\in Y_\eps + \Lambda, \\
&\psi(\bx+\be)=e^{i \bkappa \cdot \be}\psi(\bx), & \mbox{for} \; \be \in \Lambda.
\end{aligned}
\end{equation}
The eigenfunction $\psi$ is called the Bloch mode, which can be written as $\psi = e^{i \bkappa \cdot \bx}u(\bx)$, wherein $u$ is a periodic function satisfying $u(\bx+\be)=u(\bx)$ for any $\be \in \Lambda$.
Along the boundary of the obstacles, we impose the Dirichlet boundary condition
\begin{equation}\label{eq:Dirichlet}
   \psi(\bx)=0,  \quad  \bx\in \partial D_\eps + \Lambda, 
\end{equation}
or the Neumann boundary condition
\begin{equation}\label{eq:Neumann}
   \partial_\nu\psi(\bx)=0,  \quad  \bx\in \partial D_\eps + \Lambda.
\end{equation}
Here $\nu$ denotes the unit normal direction pointing to the exterior of the obstacle. We call \eqref{eq:eigen_prob}\eqref{eq:Dirichlet} and 
\eqref{eq:eigen_prob}\eqref{eq:Neumann} the Dirichlet and Neumann eigenvalue problem respectively.

Let $K=\frac{2}{3}\bkappa_1+\frac{1}{3}\bkappa_2=\frac{2\pi}{a}(\frac{1}{\sqrt3},\frac{1}{3})^T\in\cB$ and $K'=-K$ be two vertices of the Brillouin zone shown in Figure \ref{fig:honeycomb_structure} (right). The matrix
\begin{equation*}
    R =
    \begin{pmatrix}
    -1/2 & \sqrt{3}/2 \\
    -\sqrt{3}/2 & -1/2
    \end{pmatrix}
\end{equation*}
is a rotation matrix such that $R\bx$ rotates the vector $\bx$ by $2\pi/3$ clockwise on the plane. Then all vertices of the Brillouin zone $\cB$ are given by $\{K, RK, R^2K, K', RK' R^2K'\}$. In addition, the following relations hold for the reciprocal lattice vectors:
$$  R \bkappa_1 = \bkappa_2,  \quad R \bkappa_2 = - (\bkappa_1 + \bkappa_2),  \quad R(\bkappa_1 + \bkappa_2) = - \bkappa_1.$$

\begin{figure}[!htbp]
    \centering
    \includegraphics[width=7cm]{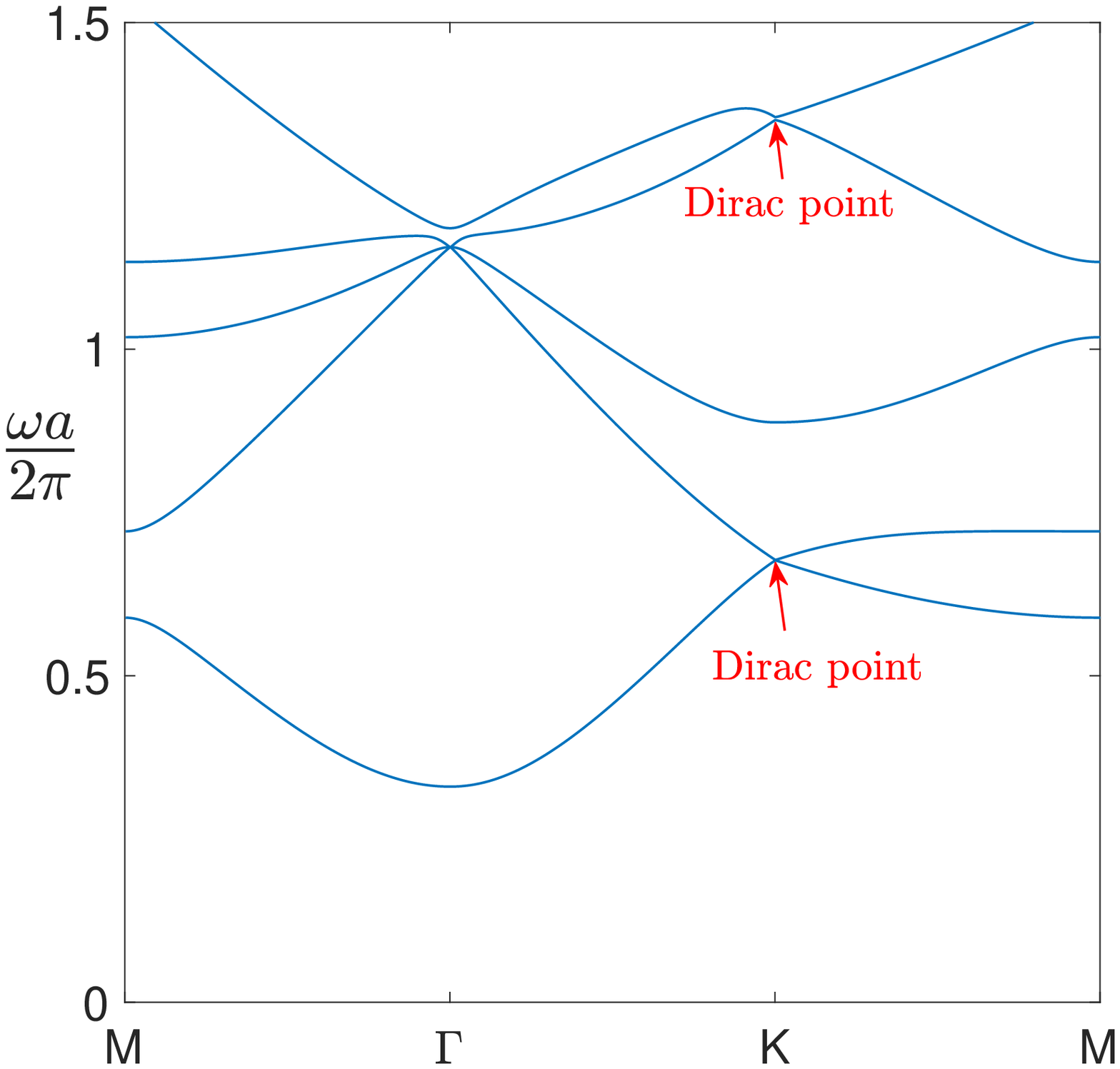}
    \includegraphics[width=9cm]{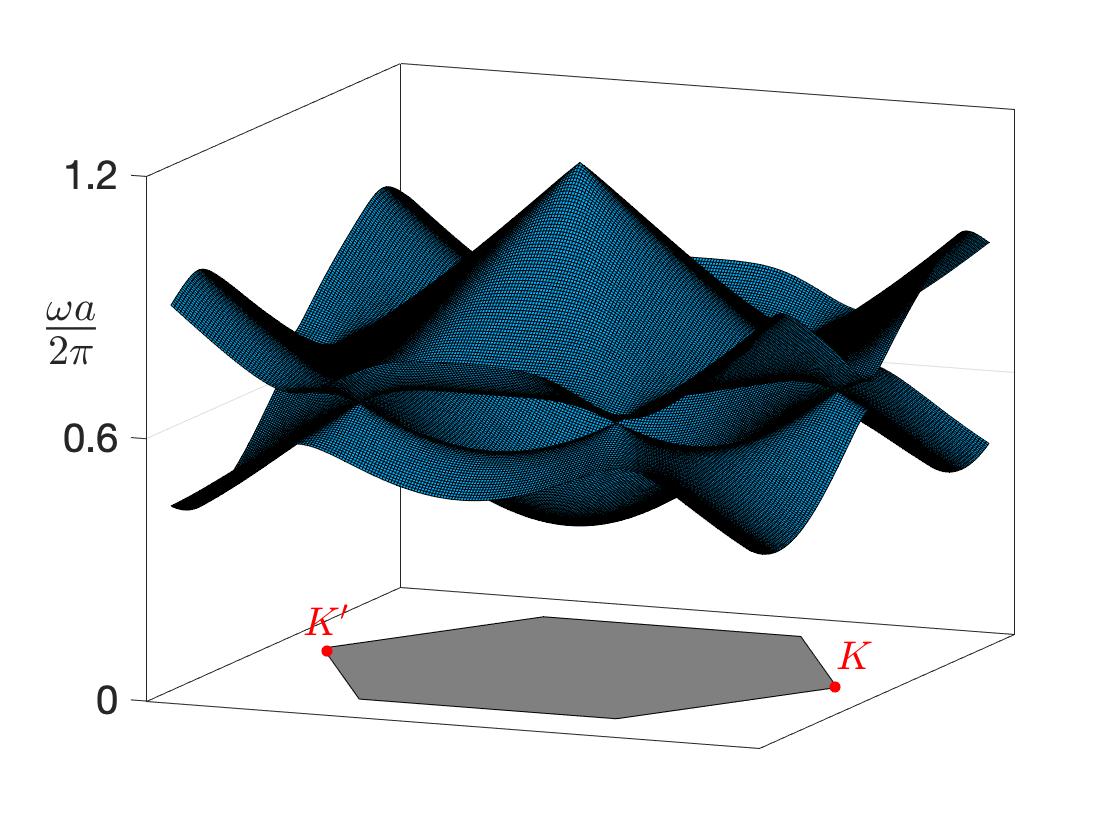}
    \caption{Left: the dispersion curves along the segments $M\to\Gamma\to K \to M$ over the Brillouin zone for the Dirichlet eigenvalue problem. The lattice constant $a=1$ and the obstacle size is $\eps=0.05$; Right: the first two dispersion surfaces for the honeycomb lattice, which shows the conical singularity at the crossing of the first two band surfaces at the high symmetry vertices of the Brillouin zone. For both the Dirichlet and Neumann eigenvalue problems, Dirac points can be formed by the crossing of the first two band surfaces or other higher band surfaces as shown on the left panel and in Section 5. }
    \label{fig:band_structure}
\end{figure}

In the sequel, we set the Bloch wave vector $\bkappa^*=K$ and investigate Dirac points at $\bkappa^*$. Dirac points located at other Bloch wave vectors are reported in Section 5, but their mathematical studies will be our future endeavors. Figure \ref{fig:band_structure} shows the occurrence of the Dirac points formed by the first and second, and the fourth and fifth band surfaces respectively for the Dirichlet eigenvalue problem. Note that due to the symmetry of the honeycomb structure, the same Dirac points also appear at other vertices of the Brillouin zone $\cB$. In this paper we prove the existence of Dirac points $(\bkappa^*, \omega^*)$ for both Dirichlet and Neumann eigenvalue problems and show that Dirac points appear at the crossing of lower band surfaces as well as higher band surfaces. In addition, we carry out quantitative analysis for the eigenvalues and the slopes of the conical dispersion surfaces near each Dirac point. It is shown that each eigenvalue $\omega^*$ is near a singular frequency associated with the Green's function for the honeycomb lattice. These singular frequencies also correspond to the eigenvalues of the homogeneous medium over the honeycomb lattice when the obstacles are absent. In addition, the slopes of the dispersion surfaces are reciprocal to the eigenvalue $\omega^*$.
We apply the layer potential technique to formulate the eigenvalue problem and reduce the integral equation to a set of characteristic equations at $\bkappa^*$ from the symmetry of the integral kernel and by the asymptotic analysis for the integral operator. The eigenvalues are roots of the nonlinear characteristic equations and we derive their asymptotic expansions with respect to the size of the obstacles $\eps$. We would like to point out that our work is closely related to \cite{ammari-20-4, Cassier-Weinstein-21} in the sense that the limit of the high-contrast elliptic operators considered in \cite{ammari-20-4, Cassier-Weinstein-21} are related to the Neumann problem investigated in Section 6, although the arrangements of inclusions considered here are different.

The rest of the paper is organized as follows. Sections 2-5 are devoted to the study of Dirac points for the Dirichlet eigenvalue problem and Section 6 discusses the Neumann eigenvalue problem. In Section 2 we formulate the eigenvalue problem by using the layer potential and set up an infinite linear system for the expansion coefficients of the density function over the obstacle boundary. The existence of the Dirac point at low frequency bands is proved and the asymptotic expansion of the corresponding eigenvalue is derived in Section 3. We establish the conical singularity of the Dirac point in Section 4 and carry out quantitative analysis for the slopes of the dispersion surfaces near this Dirac point. Finally, the Dirac points located at higher frequency bands are investigated in Section 5.

\section{An infinite linear system for the Dirichlet eigenvalue problem}
In this section, we formulate the Dirichlet eigenvalue problem by an integral equation over the obstacle boundary and set up an infinite linear system for the expansion coefficients of the density function. The Dirichlet eigenvalues reduce to the characteristic values of the infinite linear system. 

\subsection{Integral equation formulation for the eigenvalue problem}
For a given Bloch wave vector $\bkappa\in \cB$ and frequency $\omega\in\bbR$,  we use $G(\bkappa,\omega;\bx)$ to denote the corresponding quasi-periodic Green's function that satisfies 
\begin{equation}\label{eq:Green_fun}
(\Delta +\omega^2)  G(\bkappa,\omega;\bx)= \sum_{\be\in\Lambda}  e^{i \bkappa \cdot \be} \delta(\bx-\be)  \quad \bx, \by \in \bbR^2.
\end{equation}
It can be shown that
\begin{equation}\label{eq:G_lattice}
G(\bkappa,\omega;\bx) = -\frac{i}{4} \sum_{\be \in \Lambda} e^{i\bkappa\cdot \be} H_0^{(1)}(\omega |\bx-\be|),
\end{equation}
where $H_0^{(1)}$ is the zeroth-order Hankel function of the first kind. 
Alternatively, $G(\bkappa,\omega; \bx)$ adopts the following spectral representation (cf. \cite{ammari-book, ammari-kang-lee}):
\begin{equation}\label{eq:G_spec_decomp}
G(\bkappa,\omega;\bx) = \frac{1}{|Y|}\sum_{\bq\in \Lambda^*} \frac{e^{i(\bkappa+\bq)\cdot\bx}}{\omega^2-|\bkappa+\bq|^2}.
\end{equation}
Note the Green's function is not well defined when
the frequency satisfies $|\omega|=|\bkappa+\bq|$ for certain $q\in \Lambda^*$. We call such a frequency a singular frequency and denote the set of singular frequencies by
\begin{equation*}
\Omega_{\text{sing}}(\bkappa):= \{\omega: |\omega|=|\bkappa+\bq| \; \mbox{for certain} \; \bq\in\Lambda^* \}.
\end{equation*}
For each $\bkappa\in\cB$, we arrange all the singular frequencies in $\Omega_{\text{sing}}(\bkappa)$ in ascending order and denote them as
\begin{equation*}
 \bar\omega_1(\bkappa) < \bar\omega_2(\bkappa) < \cdots < \bar\omega_n(\bkappa) < \bar\omega_{n+1}(\bkappa) < \cdots.
\end{equation*}

First, the following lemma is straightforward from the expansion \eqref{eq:G_spec_decomp}.
\begin{lemma}\label{lem:G_symm1}
The Green's function $G(\bkappa,\omega;\bx)$ satisfies
\begin{equation}\label{eq:G_symm1}
G(\bkappa,\omega;\bx) = \overline{G(\bkappa,\omega;-\bx)} \quad \mbox{for} \; \bx \in\bbR^2\backslash\{0\}.
\end{equation}
\end{lemma}

\begin{lemma}\label{lem:G_symm2}
Let $\bkappa = \bkappa^*$, then the Green's function $G(\bkappa,\omega;\bx)$ satisfies
\begin{equation}\label{eq:G_symm2}
G(\bkappa,\omega;\bx) = G(\bkappa,\omega;R\bx) \quad \mbox{for} \; \bx\in\bbR^2\backslash\{0\}.
\end{equation}
\end{lemma}
\begin{proof}
For a Bloch wave vector $\bq=\ell_1 \bkappa_1 + \ell_2 \bkappa_2$, 
using the relations $R \bkappa_1 = \bkappa_2$, $R \bkappa_2 = - (\bkappa_1 + \bkappa_2)$, and $R\bkappa^*=\bkappa^*-\bkappa_1$,
it follows that
$R(\bkappa^*+\bq) = \bkappa^* - (1+\ell_2)\bkappa_1 + (\ell_1-\ell_2)\bkappa_2 = \bkappa^* + \tilde \bq$ where $\tilde \bq:=(1+\ell_2)\bkappa_1 + (\ell_1-\ell_2)\bkappa_2 \in \Lambda^*$. Note that the map from $\bq$ to $\tilde\bq$ is a bijective map on $\Lambda^*$. Therefore,
\begin{equation*}
\begin{aligned}
    G(\bkappa^*,\omega;\bx) &= \frac{1}{|Y|}\sum_{\bq\in \Lambda^*} \frac{e^{i(\bkappa^*+\bq)\cdot\bx}}{\omega^2-|\bkappa^*+\bq|^2} = \frac{1}{|Y|}\sum_{\bq\in \Lambda^*} \frac{e^{iR(\bkappa^*+\bq)\cdot R\bx}}{\omega^2-|R(\bkappa^*+\bq)|^2}\\
    &= \frac{1}{|Y|}\sum_{\tilde\bq\in \Lambda^*} \frac{e^{i(\bkappa^*+\tilde\bq)\cdot R\bx}}{\omega^2-|\bkappa^*+\tilde\bq|^2}=G(\bkappa^*,\omega;R\bx).
\end{aligned}
\end{equation*}
\end{proof}

We now introduce the following single-layer potential 
\begin{equation}\label{eq:Sdef}
[\cS^{\bkappa,\omega}\varphi](\bx):=\int_{\by\in \partial D_\eps} G(\bkappa,\omega;\bx-\by) \varphi(\by) \, ds_\by, \quad \; \bx \in   Y_\eps + \Lambda,
\end{equation}
where $\varphi$ is a density function on $\partial D_\eps$.  
Let $H^s(\partial D_\eps)$ be the standard Sobolev space of order $s$ over the boundary of $D_\eps$. It is well-known that $\cS^{\bkappa,\omega}$ is bounded from $H^{-1/2}(\partial D_\eps)$ to $H^{1/2}(\partial D_\eps)$ \cite{ammari-book, ammari-kang-lee}. We represent the Bloch mode $\psi$ for the eigenvalue problem \eqref{eq:eigen_prob}-\eqref{eq:Dirichlet} using the above defined layer potential. Using the Green's identity, it is easy to check that $(\omega,\psi)$ is an eigenpair for the Dirichlet problem \eqref{eq:eigen_prob}-\eqref{eq:Dirichlet} if and only if there exists a density function $\varphi\in H^{-1/2}(\partial D_\eps)$ such that
\begin{equation}\label{eq:int_eq1}
 [\cS^{\bkappa,\omega} \varphi](\bx)=0 \quad \mbox{for} \; \bx \in   \partial D_\eps. 
\end{equation}
To facilitate the asymptotic analysis, we apply the change of variables to rewrite the above integral equation as
\begin{equation}\label{eq:int_eq2}
[\cS_\eps^{\bkappa,\omega}\varphi](\bx)=0 \quad \mbox{for} \; \bx \in   \partial D_1,
\end{equation}
where the integral operator $\cS_\eps^{\bkappa,\omega}$ takes the form
\begin{equation}
    [\cS_\eps^{\bkappa,\omega}\varphi](\bx):= \int_{\by\in \partial D_1} G(\bkappa,\omega;\eps(\bx-\by)) \varphi(\by) \, ds_\by, \quad \bx\in\partial D_1.
\end{equation}
We seek for eigenpairs $(\omega,\varphi)$ such that \eqref{eq:int_eq2} attains nontrivial solutions.

\subsection{Eigenvalues as the characteristic values of an infinite linear system}
Let the boundary of $D_1$ be parameterized by $\br(t)=(r_1(t), r_2(t))$ where $r_1(\cdot)$ and $r_2(\cdot)$ are two smooth periodic functions with period $2\pi$. 
Then the linear operator $\cS_\eps^{\bkappa,\omega}$ induces a bounded operator from $H^{-1/2}([0,2\pi])$ to $H^{1/2}([0,2\pi])$ in the parameter space, which we still denote as $\cS_\eps^{\bkappa,\omega}$ for ease of notation.
In what follows, we shall work exclusively when $D_1$ is a unit disk and its parametric equation is given by $\br(t)=(\cos t, \sin t)$. 

We now solve the integral equation \eqref{eq:int_eq2} with the above parameterization. Define
\begin{equation}
  \phi_n(t):=\frac{1}{\sqrt{2\pi}}e^{int},\quad t\in[0,2\pi], \quad n\in\bbZ.
\end{equation}
Then $\{\phi_n\}_{n \in \bbZ}$ forms a complete orthogonal basis for $H^{-1/2}([0,2\pi])$.  We expand $\varphi\in H^{-1/2}([0,2\pi])$ as  $\displaystyle{\varphi=\sum_{n=-\infty}^\infty c_n \phi_n}$ where $\{c_{n} \}_{n\in \bbZ}  \in \mathbb{H}^{-1/2}$. 
Here and thereafter, the space $\mathbb{H}^{s}$ is defined by
\begin{equation*}
     \mathbb{H}^{s}:=\Big\{ \{c_{n} \}_{n\in \bbZ}: \sum_{n=-\infty}^{\infty} (1+n^2)^{s}|c_n|^2<\infty \Big\}.
\end{equation*} 
Then \eqref{eq:int_eq2} reads
\begin{equation*}
\sum_{n=-\infty}^\infty c_n \, \left(\cS_\eps^{\bkappa,\omega}\phi_n\right) = 0.
\end{equation*}

Define an infinite matrix $\cA=[a_{m,n}]_{m, n \in \bbZ}$, where
\begin{equation}
 a_{m,n}(\bkappa,\omega)=\left( \phi_m, \cS_\eps^{\bkappa,\omega}\phi_n \right):=\int_0^{2\pi} \overline{\phi_m(t)} \, [\cS_\eps^{\bkappa,\omega}\phi_n](t) \, dt.
\end{equation}
We see that \eqref{eq:int_eq2} holds if and only if there exists nonzero $\bc=\{c_n\}_{n\in \bbZ}\in \mathbb{H}^{-1/2}$ such that
the following infinite linear system holds:
\begin{equation}\label{eq:linear_sys2}
\cA(\bkappa,\omega)\, \bc = \bzero.
\end{equation}
Such $\omega$ is the called the characteristic values of the system. To study the eigenvalues $\omega$ of the Dirichlet problem \eqref{eq:eigen_prob}-\eqref{eq:Dirichlet}, 
we investigate the characteristic values of \eqref{eq:linear_sys2} in the rest of this paper. The matrix $\cA$ inherits the symmetries of the Green's function and the problem geometry as discussed below.

\begin{lemma}\label{lem:amn}
The following relations hold for the elements of the matrix $\cA$:
\begin{itemize}
    \item [(i)]  $a_{m,n} = \overline{a_{n,m}}$;
    \item [(ii)] $a_{m,n} = (-1)^{m-n}a_{-n,-m} = (-1)^{m-n}\overline{a_{-m,-n}} $;
    \item [(iii)] If $\bkappa=\bkappa^*$, then $a_{m,n}\neq0$ only if $\bmod(m-n,3)=0$, where $\bmod(\cdot,3)$ denotes the modulo operation with the divisor equals to $3$.

\end{itemize}
\end{lemma}

\begin{proof}
(i). A straightforward calculation yields
\begin{align*}
\overline{a_{n,m}} = \overline{\left(\phi_n, \cS_\eps^{\bkappa,\omega}\phi_m\right)} & =  \int_{-\pi}^{\pi}\int_{-\pi}^{\pi} \overline{G(\bkappa, \omega; \eps(\br(t)-\br(\tau)))} \; \overline{\phi_m(\tau)} \, d\tau \; \phi_n(t) \, dt \\
& =  \int_{-\pi}^{\pi}\int_{-\pi}^{\pi} G(\bkappa, \omega; \eps(\br(\tau)-\br(t))) \;  \phi_n(t) \, dt \; \overline{\phi_m(\tau)}  \, d\tau = a_{m,n}.
\end{align*}

(ii). Let $t=t'-\pi$ and $\tau=\tau'-\pi$, it follows that
\begin{align*}
a_{m,n}  & =  \int_{-\pi}^{\pi}\int_{-\pi}^{\pi} G(\bkappa, \omega; \eps(\br(t)-\br(\tau))) \; \phi_n(\tau) \, d\tau \; \overline{\phi_m(t)} \, dt  \\
& = \int_{0}^{2\pi}\int_{0}^{2\pi} G\big(\bkappa, \omega; \eps(\br(t'-\pi)-\br(\tau'-\pi))\big) \; \phi_n(\tau'-\pi) \, d\tau' \; \overline{\phi_m(t'-\pi)} \, dt'  \\
& = e^{i(m-n)\pi} \int_{0}^{2\pi}\int_{0}^{2\pi} G\big(\bkappa, \omega; \eps(\br(\tau')-\br(t'))\big) \;  \; \overline{\phi_m(t')} \, dt' \phi_n(\tau') \, d\tau' \\
& = (-1)^{m-n}a_{-n,-m},
\end{align*}
where we use $\phi_{-n}(t)=\overline{\phi_{n}(t)}$. \\

(iii). For $\bkappa=\bkappa^*$, using Lemma \ref{lem:G_symm2}, the integral
\begin{equation*}
   u(\bx)  := \cS_\eps^{\bkappa^*,\omega}\phi = \int_{\partial D_1} G(\bkappa^*,\omega;\bx-\by) \, \phi(\by) \,ds_\by 
     =\int_{\partial D_1} G(\bkappa^*,\omega;R\bx-\tilde\by) \, \phi\left(R^T\tilde\by\right) \,ds_{\tilde\by},
\end{equation*}
where $\tilde\by = R\by$. Setting $\phi=\phi_n$ and $u=u_n$, and using $\phi_n\left(R^T\tilde\by\right)=e^{\frac{i2n\pi}{3}}\phi_n\left(\tilde\by\right)$, we obtain
\begin{equation}\label{eq:un_rot}
    u_n(\bx) := \cS_\eps^{\bkappa^*,\omega}\phi_n = e^{\frac{i2n\pi}{3}} \int_{\partial D_1} G(\bkappa^*,\omega;R\bx-\tilde\by) \, \phi_n\left(\tilde\by\right) \,ds_{\tilde\by} = e^{\frac{i2n\pi}{3}} u_n(R\bx).
\end{equation}
On the other hand, $u_n$ attains the following expansion in the parameter space:
\begin{equation}\label{eq:un_exp}
 u_n(t)=\sum_{m=-\infty}^\infty\left(\phi_m, \cS_\eps^{\bkappa^*,\omega}\phi_n \right) \phi_m(t) =\sum_{m=-\infty}^\infty a_{m,n} \phi_m(t).
\end{equation}
Substituting into \eqref{eq:un_rot} yields
\begin{equation*}
    \sum_{m=-\infty}^\infty a_{m,n} e^{i(mt-2n\pi/3)} = \sum_{m=-\infty}^\infty a_{m,n} e^{im(t-2\pi/3)}.
\end{equation*}
Thus $a_{m,n}\neq0$  only if $mt-2n\pi/3=m(t-2\pi/3)+2m'\pi$ for some $m'\in\bbZ$, or
$m-n=3m'$.

\end{proof}

By virtue of Lemma \ref{lem:amn}, when $\bkappa=\bkappa^*$, the linear system \eqref{eq:linear_sys2} decouples into three subsystems as follows:
\begin{align}
  &  \sum_{n=-\infty}^\infty  a_{3m,3n}(\bkappa^*,\omega) \, c_{3n} = 0, \quad m=0, \pm 1, \pm 2, \cdots; \label{eq:subsys1} \\
  &   \sum_{n=-\infty}^\infty  a_{3m+1,3n+1}(\bkappa^*,\omega) \, c_{3n+1} = 0, \quad m=0, \pm 1, \pm 2, \cdots; \label{eq:subsys2} \\
  &  \sum_{n=-\infty}^\infty  a_{3m-1,3n-1}(\bkappa^*,\omega) \, c_{3n-1} = 0, \quad m=0, \pm 1, \pm 2, \cdots. \label{eq:subsys3}
\end{align}
Correspondingly, we decompose the space $\mathbb{H}^{s}$ as the direct sum $\mathbb{H}^{s} = \mathbb{H}^{s}_{\bkappa*, 0} \oplus \mathbb{H}^{s}_{\bkappa*, 1}  \oplus \mathbb{H}^{s}_{\bkappa*, -1}$, in which
\begin{equation*}
    \mathbb{H}^{s}_{\bkappa*,j} :=\Big\{ \{c_{n} \}_{n\in \bbZ}\in \mathbb{H}^{s}: c_n=0 \text{ if } \bmod(n-j,3)\neq 0 \Big\},\quad j=0, 1, -1.
\end{equation*}
Each subsystem above corresponds to restricting the full system \eqref{eq:linear_sys2} to the space $\mathbb{H}_{\bkappa*,j}^{-1/2}$.
In connection with the eigenvalue problem \eqref{eq:int_eq2}, we decompose the function space
$H^s([0,2\pi])$ into 
\begin{equation}
    H^s([0,2\pi]) = H^s_{\bkappa*,0}([0,2\pi]) \oplus H^s_{\bkappa*,1}([0,2\pi])  \oplus H^s_{\bkappa*,-1}([0,2\pi]). 
\end{equation}
Alternatively, the function space $H^s_{\bkappa^*,j}([0,2\pi])$ can be characterized as follows:
\begin{align*}
H^s_{\bkappa*,0}([0,2\pi]) &= \Big\{ \phi(t) \in H^s([0,2\pi]): \phi\left(t+\frac{2\pi}{3}\right) = \phi(t) \Big\}, \\
H^s_{\bkappa*,1}([0,2\pi]) &= \Big\{ \phi(t) \in H^s([0,2\pi]): \phi\left(t+\frac{2\pi}{3}\right) = e^{i\frac{2\pi}{3}}\phi(t) \Big\}, \\
H^s_{\bkappa*,-1}([0,2\pi]) &= \Big\{ \phi(t) \in H^s([0,2\pi]): \phi\left(t+\frac{2\pi}{3}\right) = e^{-i\frac{2\pi}{3}}\phi(t) \Big\}.
\end{align*}
If $\{c_{3n+j} \}_{n \in \bbZ} \in \mathbb{H}^{-1/2}$ is an eigenvector for the corresponding system in \eqref{eq:subsys1}-\eqref{eq:subsys3}, then
the eigenfunction $\displaystyle{\phi(t) = \sum_{n=-\infty}^{\infty}c_{3n+j} \phi_{3n+j}(t)}$ for the eigenvalue problem \eqref{eq:int_eq2} belongs to $H^{-1/2}_{\bkappa*,j}([0,2\pi])$.

In the sequel, we investigate the characteristic values $\omega$ for each  of \eqref{eq:subsys1}-\eqref{eq:subsys3} such that the system attains nontrivial solutions $\{c_{3n+j}\}_{n\in\bbZ}$ for $j=0, 1$ or $-1$. 
As shown below, the systems \eqref{eq:subsys2} and \eqref{eq:subsys3} attain the same characteristic values.

\begin{proposition}\label{prop_repeated_eig}
$\omega$ is a characteristic value of \eqref{eq:subsys2} with the corresponding solution $\{c_{3n+1} \}_{n\in \bbZ}$ if and only if
$\omega$ is a characteristic value of \eqref{eq:subsys3} with the solution $\{c_{3n-1} \}_{n\in\bbZ}$ satisfying $c_{3n-1}=(-1)^{-3n}\overline{c_{-3n+1}}$ for each $n$.
\end{proposition}

\begin{proof}
Let $\omega$ be a characteristic value of \eqref{eq:subsys2} and $\{c_{3n+1} \}_{n\in \bbZ}$ be the corresponding solution such that
\begin{equation*}
    \sum_{n=-\infty}^\infty  a_{3m+1,3n+1}(\bkappa^*,\omega) \, c_{3n+1} = 0, \quad m=0, \pm 1, \pm 2, \cdots. \label{eq:subsys2-2} 
\end{equation*}
Choose $c_{3n-1}=(-1)^{-3n}\overline{c_{-3n+1}}$ for $n\in\bbZ$. Then by Lemma \ref{lem:amn} (ii), for each $m\in \bbZ$, we have
\begin{align*}
    \sum_{n=-\infty}^\infty a_{3m-1,3n-1} (\bkappa^*,\omega) \, c_{3n-1} &= \sum_{n=-\infty}^\infty a_{3m-1,-3n-1} (\bkappa^*,\omega) \, c_{-3n-1} \\
    & =  \sum_{n=-\infty}^\infty (-1)^{3m+3n} \overline{a_{-3m+1,3n+1}} (\bkappa^*,\omega) \cdot (-1)^{3n}  \overline{c_{3n+1}} = 0,
\end{align*}
and the system \eqref{eq:subsys3} holds. The converse can be shown similarly.
\end{proof}

\section{Dirichlet eigenvalue at $\bkappa=\bkappa^*$ for the low-frequency bands}\label{sec:omega1}
In this section, we focus on the lowest eigenvalue to the eigenvalue problem \eqref{eq:eigen_prob}-\eqref{eq:Dirichlet} when $\bkappa=\bkappa^*$. Based on the decomposition of the quasi-periodic Green's function $G(\bkappa, \omega, \bx)$ and the integral operator $\cS_\eps^{\bkappa,\omega}$, we decompose the matrix $\cA$ as $\cA=\cD+ \eps \, \cE$, wherein $\cD$ is a diagonal matrix. Such decomposition allows for reducing the subsystems \eqref{eq:subsys1}-\eqref{eq:subsys3} to three scalar nonlinear equations (characteristic equations). The eigenvalues are the roots of the characteristic equations and can be obtained by the asymptotic analysis.

\subsection{Decomposition of the Green's function and the single-layer operator $ \cS_\eps^{\bkappa,\omega}$ }
As we shall see, the lowest Dirichlet eigenvalue at $\bkappa=\bkappa^*$ lies in the neighborhood of the singular frequency $\bar\omega_1(\bkappa^*):=|\bkappa^*|$. To this end, we consider the following small neighborhood of $\bar\omega_1(\bkappa^*)$: 
\begin{equation}
    \Omega_\eps(\bkappa^*): = \left\{ \omega \in \mathbb R^+: \frac{1}{\abs{2Y}} \left(\frac{2\pi}{a}\right)^2\eps^2 \le \omega^2-\abs{\bkappa^*}^2 \le \dfrac{4\pi}{\abs{Y}\abs{\ln\eps}} \right\}.
\end{equation}
The lower and upper bounds are chosen such that there exist roots $\omega$ for the systems \eqref{eq:subsys1}-\eqref{eq:subsys3} in $\Omega_\eps(\bkappa^*)$.
Let $U_r:=\left\{\bx:|\bx|< r \right\}$ be disk centered at the origin with the radius $r$.

We denote
\begin{equation}\label{eq:Lstar1}
    \Lambda^*_0(\bar\omega_1):=\left\{\bq\in\Lambda^*: |\bkappa^*+\bq|=|\bkappa^*|\right\}.
\end{equation}
It can be solved that $\Lambda^*_0(\bar\omega_1)=\left\{\bq_1,\bq_2,\bq_3\right\}$,
where $\bq_1=(0,0)^T$, $\bq_2=\frac{2\pi}{a}(-\frac{2}{\sqrt3},0)^T$, and $\bq_3=\frac{2\pi}{a}(-\frac{1}{\sqrt3},-1)^T$.
We decompose the Green's function into the three parts as follows:
\begin{definition}[Decomposition of the Green's function]
Let
\begin{equation}\label{eq:GH}
 H_0(\omega; \bx):= -\frac{i}{4}  H_0^{(1)}(\omega|\bx|), \quad \bx \neq 0,
\end{equation}
\begin{equation}\label{eq:GL}
 \GL(\bkappa, \omega;\bx):=\frac{1}{|Y|}\sum_{\bq\in \Lambda_0^*(\bar\omega_1)} \frac{e^{i(\bkappa+\bq)\cdot\bx}}{\omega^2-|\bkappa+\bq|^2},  
\end{equation}
and 
\begin{eqnarray}
 \tG(\bkappa, \omega; \bx) &:=& G(\bkappa,\omega;\bx)-H_0(\omega;\bx)-\GL(\bkappa,\omega;\bx), \quad \bx \neq 0, \label{eq:tG} \\
  \tG(\bkappa, \omega; 0) &:=& \lim_{\bx \to 0 } \tG(\bkappa, \omega; \bx). \nonumber
\end{eqnarray}
\end{definition}

\begin{remark}
$H_0(\omega; \bx)$ is the free-space Green's function that satisfies $(\Delta +\omega^2)H_0(\bx)=\delta(\bx)$. Its asymptotic behavior for  $0< |\bx| \ll 1$ is well known and is given in the next lemma. In particular, $H_0(\omega; \bx) \approx \frac{1}{2\pi}\ln|\bx|$ as $|\bx|\to 0$.
\end{remark}
\begin{remark}
Given non-singular frequency $\omega \notin \Omega_{\text{sing}}(\bkappa)$, both $\GL(\bkappa, \omega;\bx)$ and $ \tG(\bkappa, \omega; \bx)$ are smooth functions in the neighborhood of $\bx=0$. However, their asymptotic behaviors are very different as $\omega$ approaches the singular frequency $\bar\omega_1(\bkappa^*)=|\bkappa^*|$.
More precisely, in this region, the finite-sum $\GL(\bkappa^*, \omega;\bx)$ attains the order $\frac{1}{\omega^2-|\bkappa^*|^2}$ and its value blows up as $\omega\to\bar\omega_1(\bkappa^*)$, while  $\tG(\bkappa^*,\omega;\bx)$ remains order $O(1)$ in the neighborhood of $\bx=0$ as $\omega\to \bar\omega_1$. In the decomposition,
we introduce $\GL$ to extract the singular behavior of the Green's function when $\omega$ is close to the singular frequency.
\end{remark}

\begin{lemma} [\cite{hsiao-wendland}, Section 2.1.1]
If $0< |\bx| \ll 1$, then
 \begin{equation}\label{eq:H0form}
H_0(\omega;\bx) =\frac{1}{2\pi}\left[\ln|\bx|+\ln\omega+\gamma_0+\ln(\omega|\bx|)\sum_{p\geq1}b_{p,1}(\omega|\bx|)^{2p}+\sum_{p\geq1}b_{p,2}(\omega|\bx|)^{2p}
 \right],
 \end{equation}
 where
 \begin{equation}
    b_{p,1} = \frac{(-1)^p}{2^{2p}(p!)^2}, \;b_{p,2} = \left(\gamma_0 - \sum_{s=1}^p\frac{1}{s}\right) b_{p,1},  \;  \gamma_0 = E_0\footnote{$E_0$ is the Euler constant given by $\displaystyle{E_0=\lim_{N\to\infty}\left(\sum_{p=1}^N \frac{1}{p}-\ln N\right)}$.} - \ln 2 - \frac{i\pi}{2}.
\end{equation}
\end{lemma}

\medskip

From the Taylor expansion of the finite-sum $\GL(\bx)$, we also have the following lemma.
\begin{lemma}\label{lem:tG_analyticity}
For each $\omega \in \Omega_\eps(\bkappa^*)$, $\GL(\bkappa^*,\omega;\bx)$ is analytic for $\bx\in U_{2\eps}$ and it possesses the Taylor expansion
\begin{equation}\label{eq:GL_exp}
     \GL(\bkappa^*,\omega;\bx)=\frac{1}{|Y|}\frac{1}{\omega^2-|\bkappa^*|^2} \left( 3 - \frac{1}{3}\left(\frac{2\pi}{a}\right)^2 |\bx|^2\right)+\GL^\infty(\bkappa^*,\omega;\bx),
\end{equation}
where 
\begin{equation}
    \GL^\infty(\bkappa^*,\omega;\bx)=\frac{1}{|Y|}\frac{1}{\omega^2-|\bkappa^*|^2} \sum_{|\alpha|\geq3}c_{\alpha}(\omega)x_1^{\alpha_1}x_2^{\alpha_2},
\end{equation} 
$\alpha=(\alpha_1,\alpha_2)$ and $|c_\alpha| < C^{|\alpha|}$ for certain constant $C$ independent of $\omega$, $\eps$ and $\alpha$.
\end{lemma}

\begin{lemma}\label{lem:tG_analyticity}
For each $\omega \in \Omega_\eps(\bkappa^*)$, 
$\tG(\bkappa^*,\omega;\bx)$ is smooth for $\bx\in U_{2\eps}$.
In addition,
\begin{equation}\label{eq:tG_bnd}
\sup_{\omega\in \Omega_\eps(\bkappa^*)} \left|\partial^{\alpha_1}_{x_1}\partial^{\alpha_2}_{x_2}\tG(\bkappa^*,\omega;0)\right|\leq C,  \quad 0 \le \alpha_1+\alpha_2 \le 2,
\end{equation}
wherein the constant $C$ is independent of $\omega$, $\eps$.
\end{lemma}
\begin{proof}
For fixed $\bkappa^*$, from the spectral representation of the Green's function \eqref{eq:G_spec_decomp}, we see that $G(\bkappa^*,\omega;\cdot) - \GL(\bkappa^*,\omega;\cdot)$ is a family of distribution that depends on $\omega$ analytically for $\omega \in \Omega_\eps(\bkappa^*)$. So is the distribution $\tG(\bkappa^*,\omega, \cdot)=G(\bkappa^*,\omega, \cdot) - \GL(\bkappa^*,\omega, \cdot)-H_0(\cdot)$. 
On the other hand, in view of \eqref{eq:Green_fun}, there holds
$(\Delta +\omega^2)  \tG(\bkappa^*,\omega, \bx) = -\frac{3}{|Y|}$ for $\bx\in U_{2\eps}$ and $\omega \in \Omega_\eps(\bkappa^*)$, where we used the fact that 
there are three elments in the set $\Lambda^*_0(\bar\omega_1)$.
From the regularity theory for the solutions to the Helmholtz equation, we deduce that the distribution $\tG(\bkappa^*,\omega;\bx)$ is smooth in the domain $U_{2\eps}$. Hence $\tG(\bkappa^*,\omega;\bx)$ can be viewed as a family of smooth functions for $\bx\in U_{2\eps}$ that depends on the parameter $\omega$ analytically. This completes the proof of the lemma. 
\end{proof}

\begin{definition}[Decomposition of the single-layer operator $ \cS_\eps^{\bkappa,\omega}$]
Let $\SH$, $\SL$ and $\tS$ be the integral operators with the kernel $H_0$, $\GL$ and $\tG$ given in \eqref{eq:GH}-\eqref{eq:tG}:
\begin{equation*}
\begin{aligned}
    [\SH\varphi](\bx) :=& \int_{\by\in \partial D_1} H_0(\omega;\eps(\bx-\by)) \varphi(\by) \, ds_\by, \quad \bx\in\partial D_1, \\
   [\SL\varphi](\bx) :=& \int_{\by\in \partial D_1} \GL(\bkappa,\omega;\eps(\bx-\by)) \varphi(\by) \, ds_\by, \quad \bx\in\partial D_1,  \\
   [\tS\varphi](\bx) :=& \int_{\by\in \partial D_1} \tG(\bkappa,\omega;\eps(\bx-\by)) \varphi(\by) \, ds_\by, \quad \bx\in\partial D_1.  
\end{aligned}
\end{equation*}
\end{definition}

\subsection{Decomposition of the matrix $\cA$}
Recall that $a_{m,n}=\left( \phi_m, \cS_\eps^{\bkappa,\omega}\phi_n \right)$, using the decomposition of the integral operator $\cS_\eps^{\bkappa,\omega}$, we express $a_{m,n}$ as the sum of the following three terms:
\begin{equation}
\begin{aligned}
   (\SH)_{m,n}:=&\left( \phi_m, \SH\phi_n \right),\\
   (\SL)_{m,n}:=&\left( \phi_m, \SL\phi_n \right),\\
   (\tS)_{m,n}:=&\left( \phi_m, \tS\phi_n \right).\\
\end{aligned}
\end{equation}
In what follows, we obtain the asymptotic expansion of $(\SH)_{m,n}$, $(\SL)_{m,n}$, and $(\tS)_{m,n}$ to obtain a decomposition of the matrix $\cA$.

Define 
\begin{equation}
    \cS_0 \varphi(\bx) := \frac{1}{2\pi}\int_{\by\in \partial D_1} \ln|\bx-\by| \, \varphi(\by) \, ds_\by, \quad \bx\in\partial D_1. \label{eq:S0}
\end{equation}
\begin{lemma}\label{lem:S0}
The operator $\cS_0$ is bounded from $H^{-1/2}([0,2\pi])$ to $H^{1/2}([0,2\pi])$, and attains the eigenvalues $\{ \eta_n\}_{n=-\infty}^{\infty}$ and the eigenfunctions $\{ \phi_n\}_{n=-\infty}^{\infty}$ given by
\begin{equation*}
\eta_n = \left\{ 
\begin{array}{cc}
0, & n = 0;\\
-\frac{1}{2|n|}, & n\neq0; 
\end{array}
\right.
\quad \mbox{and} \quad \phi_n = \frac{1}{\sqrt{2\pi}}e^{int}.
\end{equation*}
\end{lemma}

\begin{proof}
On the unit circle,  there holds $\abs{\br(t)-\br(\tau)}^2 = 2-2\cos(t-\tau)=4\sin^2\left(\frac{t-\tau}{2}\right)$.
As such
\begin{equation*}
    [\cS_0 \phi_n](t) = \frac{1}{4\pi} \int_0^{2\pi} \ln\left(4\sin^2\left(\frac{t-\tau}{2}\right)\right) \phi_n(\tau) \, d\tau = \eta_n \phi_n(t),
\end{equation*}
where we use Lemma 8.23 in \cite{kress}.
\end{proof}

\begin{lemma}\label{lem:est1}
For $\eps\ll1$, there holds
\begin{equation}\label{eq:SHest}
   (\SH)_{m,n}=
   \begin{cases}
   \dfrac{1}{2\pi}(\ln\eps+\ln\omega+\gamma_0)+O(\eps^2\ln\eps) \quad &m=n=0, \\
   -\dfrac{1}{2|n|}\left(1+O(\eps^2\ln\eps)\right)\quad &m=n\neq0,\\
   0 \quad &m\neq n.
   \end{cases}
\end{equation}
\end{lemma}
\begin{proof}

Using the expansion \eqref{eq:H0form}, we have
\begin{equation*}
H_0(\omega,\varepsilon(\bx-\by)) =\frac{1}{2\pi}\left(\ln\varepsilon+\ln\omega+\gamma_0+\ln|\bx-\by| \right) +  H_0^\infty(\omega;\eps(\bx-\by)),
\end{equation*}
wherein
\begin{equation}
    H_0^\infty(\omega;\eps(\bx-\by)):=\frac{1}{2\pi}\left[\ln(\omega\eps|\bx-\by|)\sum_{p\geq1}b_{p,1}(\omega\eps|\bx-\by|)^{2p}+\sum_{p\geq1}b_{p,2}(\omega\eps|\bx-\by|)^{2p}\right].
\end{equation}
Note that $|\br(t)-\br(\tau)|^2=4\sin^2\left(\frac{t-\tau}{2}\right)$, we obtain $(\SH)_{m,n}=0$ for $m\neq n$.

To estimate $(\SH)_{n,n}$, from Lemma \ref{lem:S0} it is straightforward that
\begin{equation}\label{eq:phin_ln_phin}
    \frac{1}{2\pi}\int_0^{2\pi}\int_0^{2\pi}\overline{\phi_n(t)}\ln|\br(t)-\br(\tau)| \phi_n(\tau)d\tau dt=\eta_n.
\end{equation}
Now consider the following two integrals for $p\geq1$:
\begin{equation}
\begin{aligned}
    I_{1,p,n}:=&\int_0^{2\pi}\int_0^{2\pi}\overline{\phi_n(t)}|\br(t)-\br(\tau)|^{2p}\ln|\br(t)-\br(\tau)| \phi_n(\tau)d\tau dt\\
    =&\int_0^{2\pi}\frac{1}{2}\left(2-2\cos t\right)^p\ln\left(2-2\cos t\right)e^{int}\frac{dt}{2\pi}; \\
    I_{2,p,n}:=& \int_0^{2\pi}\int_0^{2\pi}\overline{\phi_n(t)}|\br(t)-\br(\tau)|^{2p} \phi_n(\tau)d\tau dt=\int_0^{2\pi}\left(2-2\cos t\right)^pe^{int}\frac{dt}{2\pi}.
    \end{aligned}
\end{equation}
When $n=0$, there holds
\begin{equation*}
    |I_{1,p,0}| \leq \frac{4^p}{2}\int_0^{2\pi}|\ln\left(2-2\cos t\right)|\frac{dt}{2\pi}= C_1\frac{4^p}{2}, \quad
    |I_{2,p,0}| \leq 4^p.
\end{equation*}
Here 
 $  \displaystyle {C_1:=\int_0^{2\pi}|\ln\left(2-2\cos t)\right)|\frac{dt}{2\pi} } $ is a finite constant.
When $n\neq0$, integrating by parts yields
\begin{equation*}
\begin{aligned}
    I_{1,p,n}&=-\frac{1}{2in}\int_0^{2\pi}[2p\sin t\left(2-2\cos t\right)^{p-1}\ln\left(2-2\cos t\right)+2\sin t\left(2-2\cos t\right)^{p-1}]e^{int}\frac{dt}{2\pi}, \\
    I_{2,p,n}&=-\frac{1}{in}\int_0^{2\pi}2p\sin t\left(2-2\cos t\right)^{p-1}e^{int}\frac{dt}{2\pi}.
\end{aligned}
\end{equation*}
It follows that
\begin{equation*}
    |I_{1,p,n}|\leq\frac{1}{2|n|}(2C_1p+2)4^{p-1}, \quad
    |I_{2,p,n}|\leq\frac{p}{2|n|}4^p.
\end{equation*}
Thus there exists a constant $C_2>0$ such that for all $0<\eps<\min(\frac{a}{4},\frac{3}{4|\bkappa^*|})$ and all $\omega\in \Omega_\eps(\bkappa^*)$,
\begin{equation}
\begin{aligned}\label{eq:phin_H0_inf_phin}
   \left|(\phi_m,H_0^\infty \phi_n)\right|&=\left|\frac{1}{2\pi}\sum_{p\geq1}(\eps\omega)^{2p}b_{1,p}I_{1,p,n}
   +\frac{1}{2\pi}\sum_{p\geq1} \left[\ln(\eps\omega) +\gamma_0-\sum_{s=1}^p\frac{1}{s}\right] (\eps\omega)^{2p}b_{1,p}I_{2,p,n}\right|\\
   &\leq
   \begin{cases}
   C_2\eps^2\ln\eps\quad &n=0, \\
   \frac{1}{2|n|}C_2\eps^2\ln\eps\quad &|n|\neq0.
   \end{cases}
\end{aligned}
\end{equation}
The proof is complete by combining \eqref{eq:phin_ln_phin} and \eqref{eq:phin_H0_inf_phin}.
\end{proof}

\begin{lemma}\label{lem:est2}
Let $\bkappa=\bkappa^*$. For $\eps\ll 1$ and $\omega\in \Omega_\eps(\bkappa^*)$, there holds
\begin{equation}\label{eq:SLest}
   (\SL)_{m,n}=
   \begin{cases}
   \dfrac{1}{|Y|}\dfrac{1}{\omega^2-|\bkappa^*|^2}\left(3- \dfrac{4\pi\eps^2}{3}\left(\dfrac{2\pi}{a}\right)^2+O(\eps^3)\right) \quad &m=n=0, \\
   \dfrac{1}{|Y|}\dfrac{1}{\omega^2-|\bkappa^*|^2}
  \left(\dfrac{2\pi\eps^2}{3}\left(\dfrac{2\pi}{a}\right)^2+O(\eps^3)
  \right) \quad &m=n=\pm1\\
   \dfrac{1}{|Y|}\dfrac{1}{\omega^2-|\bkappa^*|^2}\, O\left(\eps^{\max(3,|m|,|n|)}\right) \quad &\text{otherwise}.
   \end{cases},
\end{equation}
\end{lemma}

\begin{proof}
Using the expansion \eqref{eq:GL_exp}, there holds
\begin{equation*}
     \GL(\bkappa^*,\omega;\eps(\bx-\by))=\frac{1}{|Y|}\frac{1}{\omega^2-|\bkappa^*|^2} \left( 3 - \frac{\eps^2}{3}\left(\frac{2\pi}{a}\right)^2 |\bx-\by|^2\right)+\GL^\infty(\bkappa^*,\omega;\eps(\bx-\by)).
\end{equation*}
Then \eqref{eq:SLest} follows by using the relation $|\br(t)-\br(\tau)|^2=2-2\cos(t-\tau)$
and the fact
\begin{equation*}
    \left|\int_0^{2\pi}\int_0^{2\pi}\overline{\phi_m(\tau)}(\sin t-\sin\tau)^{\alpha_1}(\cos t-\cos\tau)^{\alpha_2}\phi_n(t)d\tau dt\right|\leq 
    \begin{cases}
    0\; &|m|\geq|\alpha| \text{ or }|n|\geq|\alpha|, \\
    2\pi\cdot 4^{|\alpha|}\; &\text{otherwise}.
    \end{cases}
\end{equation*}
\end{proof}

\begin{lemma}\label{lem:est3}
Let $\bkappa=\bkappa^*$. For $\eps\ll1$ and $\omega\in \Omega_\eps(\bkappa^*)$, $(\tS)_{m,n}$ can be expressed as
\begin{equation}\label{eq:tSest}
   (\tS)_{m,n}=
   \begin{cases}
   \tG(\bkappa^*,\omega;\mathbf{0})+\eps \cdot \tilde{a}_{0,0} \quad &m=n=0, \\
   \eps \cdot \tilde{a}_{m,n} \quad &\text{otherwise},
   \end{cases}
\end{equation}
where the operator $\tilde\cA=[\tilde{a}_{m,n}]$ is bounded from $\mathbb{H}^{-1/2}$ to $\mathbb{H}^{1/2}$, and the operator norm $\norm{\tilde\cA}\le C$ with $C$ independent of $\eps$ and $\omega$.
\end{lemma}
\begin{proof}
From the analyticity of $\tG(\bkappa^*,\omega;\bx)$, for each $\omega\in \Omega_\eps(\bkappa^*)$,
we can write $\tG(\bkappa^*,\omega;\eps(\bx-\by))$ as
\begin{equation*}
     \tG(\bkappa^*,\omega;\eps(\bx-\by))=\tG(\bkappa^*,\omega;\mathbf{0})+ \eps \cdot  \tG_\infty(\bkappa^*,\omega;\eps(\bx-\by))
\end{equation*}
for certain function
$\tG_\infty(\bkappa^*,\omega;\eps(\bx-\by))$ that is smooth for $\bx, \by \in\partial D_1$. 
In addition, from \eqref{eq:tG_bnd}, $ \tG_\infty(\bkappa^*,\omega;\eps(\bx-\by))$ together with its first order partial derivatives with respect to $\bx$, $\by$ are all uniformly bounded for $\omega\in \Omega_\eps(\bkappa^*)$. Therefore the following operator 
\begin{eqnarray*}
 \tilde{\cS}_{\eps,\infty}\varphi(\bx) := \int_{\by\in \partial D_1} \tG_\infty (\bkappa,\omega;\eps(\bx-\by)) \varphi(\by) \, ds_\by, \quad \bx\in\partial D_1,
\end{eqnarray*}
is bounded from $H^{-1/2}(\partial D_1)$ to $H^{1/2}(\partial D_1)$. 
Let
$\tilde{a}_{m,n} := \left( \phi_m, \tilde{\cS}_{\eps,\infty}\phi_n \right)$,
with $\tilde{\cS}_{\eps,\infty}$ being the operator above in the parameter space.
Then $\tilde{A}=[\tilde{a}_{m,n}]$ is bounded from $\mathbb{H}^{-1/2}$ to $\mathbb{H}^{1/2}$. This completes the proof of the lemma. 
\end{proof}

Note that for $\omega\in \Omega_\eps(\bkappa^*)$,
\begin{equation*}
   C_1|\ln\eps| \leq\frac{1}{\omega^2-|\bkappa^*|^2}\leq C_2\frac{1}{\eps^2}.
\end{equation*}
Therefore, by virtue of Lemmas \ref{lem:est1} - \ref{lem:est3}, we obtain the following decomposition for the matrix $\cA$:
\begin{proposition}[Decomposition of $\cA$] \label{prop:decomp_A}
Let $\bkappa=\bkappa^*$. There exists a constant $\sfc>0$ such that for $\eps\in (0,\sfc)$ and $\omega\in \Omega_\eps(\bkappa^*)$, the matrix $\cA$ can be decomposed as $\cA=\cD+ \eps \, \cE$, where
$\cD:=\diag(d_n)_{n\in\mathbb Z}$ with
\begin{equation*}
   d_n=
   \begin{cases}
  \dfrac{1}{2\pi}(\ln\eps+\ln\omega+\gamma_0))+\dfrac{1}{|Y|}\dfrac{1}{\omega^2-|\bkappa^*|^2}\left(3- \dfrac{4\pi\eps^2}{3}\left(\dfrac{2\pi}{a}\right)^2\right)+\tG(\bkappa^*,\omega;\mathbf{0}) & n=0, \\
    -\dfrac{1}{2}+\dfrac{1}{|Y|}\dfrac{1}{\omega^2-|\bkappa^*|^2}\dfrac{2\pi\eps^2}{3}\left(\dfrac{2\pi}{a}\right)^2, & n=\pm1, \\
   -\dfrac{1}{2|n|}, & |n|>1,
   \end{cases}
\end{equation*}
and $\cE=[e_{m,n}]$. In addition, $\cE$ is bounded from $\mathbb{H}^{-1/2}$ to $\mathbb{H}^{1/2}$ and $\norm{\cE}\le C$ for certain constant $C$ independent of $\eps$ and $\omega$.
\end{proposition}

\subsection{Characteristic equations}
We reduce the subsystems \eqref{eq:subsys1}-\eqref{eq:subsys3} to the nonlinear characteristic equations for $\omega$ by using the decomposition of the matrix $\cA$ in Proposition \ref{prop:decomp_A}. To this end, we denote 
$\bbZ^* = \bbZ\backslash \{0\}$ and
define the vectors
\begin{eqnarray}
&& \hat\ba_0 := \{a_{3m,0} \}_{m\in \bbZ^* }, \; \hat\ba_1 := \{a_{3m+1,1} \}_{m\in \bbZ^*}, \; \hat\ba_{-1} := \{a_{3m-1,-1} \}_{m\in \bbZ^*}, \label{eq:hat_a}\\
&& \hat\bc_0 := \{c_{3m} \}_{m\in \bbZ^*}, \; \hat\bc_1 := \{c_{3m+1} \}_{m\in \bbZ^*}, \; \hat\bc_{-1} := \{c_{3m-1} \}_{m\in \bbZ^*}, \label{eq:hat_c}
\end{eqnarray}
and matrices
\begin{equation}\label{eq:hat_Aj}
\hat\cA_0:=[a_{3m,3n}]_{m\in \bbZ^*,n\in \bbZ^*}, \; \hat\cA_1:=[a_{3m+1,3n+1}]_{m\in \bbZ^*,n\in \bbZ^*}, \mbox{and} \; \hat\cA_{-1}:=[a_{3m-1,3n-1}]_{m\in \bbZ^*,n\in \bbZ^*}.   
\end{equation}
Then using Lemma \ref{lem:amn} (i), each system in \eqref{eq:subsys1}-\eqref{eq:subsys3} can be split into the following two equations:
\begin{equation}\label{eq:subsys_compact}
a_{j,j} \, c_j + \left\langle \hat\bc_j,\hat\ba_j \right\rangle \,  = 0, \quad  \hat\cA_j \hat\bc_j + c_j \hat\ba_j = \bzero, \quad j = 0, 1, -1,
\end{equation}
where the equation for $m=0$ and $m=\pm1, \pm2, \cdots$ are treated separately.
Here and thereafter, the inner product $\displaystyle{\langle \ba,\bb \rangle := \sum_{m=-\infty}^\infty a_m \bar b_m}$ for the vectors $\ba := \{a_{m} \}_{m\in \bbZ}$ and $\bb := \{b_{m} \}_{m\in \bbZ}$.

\medskip

\begin{theorem}\label{thm:A_inv}
For $\bkappa=\bkappa^*$, $\omega\in \Omega_\eps(\bkappa^*)$ and sufficiently small $\eps$, the operator $\hat\cA_j: \mathbb{H}^{-1/2} \to \mathbb{H}^{1/2}$ is invertible and the operator norm $\norm{\hat\cA_j^{-1}}<1/2$.  
\end{theorem}

\begin{proof}
For each $\omega \in  \Omega_\eps(\bkappa^*)$,
let $\cA$ be decomposed as $\cA=\cD+\eps\cE$ as in Proposition \ref{prop:decomp_A}.
Similarly, we decompose $\hat\cA_j$ as
$  \hat\cA_j=\hat\cD_j+\eps\hat\cE_j$, in which $ \hat\cD_j:=\diag(d_{3n+j})_{n\neq0}$ is a diagonal matrix and the matrix $\hat\cE_j=[e_{3m+j,3n+j}]$ is bounded from $\mathbb{H}^{-1/2}$ to $\mathbb{H}^{1/2}$.

It is obvious that $\hat\cD_j$ is bounded from $\mathbb{H}^{-1/2}$ to $\mathbb{H}^{1/2}$.
In addition, the inverse of $\hat\cD_j$ exists and  $\hat\cD_j^{-1}:=\diag(1/d_{3n+j})_{n\neq0}$ is bounded from $\mathbb{H}^{1/2}$ to $\mathbb{H}^{-1/2}$, with the operator norm bounded by $4$. Let us express $\hat\cA_j$ as
\begin{equation*}
   \hat\cA_j=\hat\cD_j(I+\eps\hat\cD_j^{-1}\hat\cE_j).
\end{equation*}
For sufficiently small $\eps$, $\eps\,\hat\cD_j^{-1}\hat\cE_j$ is bounded on $\mathbb{H}^{-1/2}$, with the operator norm bounded by $1/2$.
Hence $I+\hat\cD_j^{-1}\hat\cE_j$ is an invertible operator on $\mathbb{H}^{-1/2}$, with norm bounded by $1/2$. We conclude that
$\hat\cA_j$ attains the inverse $\hat\cA_j^{-1}=(I+\eps\hat\cD_j^{-1}\hat\cE_j)^{-1}\hat\cD_j^{-1}$, with
$\|\hat\cA_j^{-1}\|<1/2$
for $\omega \in  \Omega_\eps(\bkappa^*)$.
\end{proof}

\medskip

Now by Theorem \ref{thm:A_inv}, we can express $\hat\bc_j$ as
\begin{equation}\label{eq:bcj}
\hat\bc_j = -c_j (\hat\cA_j ^{-1} \hat\ba_j),  \quad j = 0, 1, -1. 
\end{equation}
Substituting into the equation for $m=0$ in \eqref{eq:subsys_compact}, we obtain the following three equations for $c_j$:
\begin{equation}\label{eq:cj}
    a_{j,j}(\bkappa^*,\omega) \, c_j - \left\langle \hat\cA_j^{-1}\hat\ba_j(\bkappa^*,\omega),\hat\ba_j(\bkappa^*,\omega) \right\rangle \, c_j = 0, \quad j = 0, 1, -1.
\end{equation}
To obtain the eigenvalues for $\bkappa=\bkappa^*$, we solve for $\omega$ such that \eqref{eq:cj} attains nontrivial solutions, or equivalently, we find $\omega$ that is a root of one of the  characteristic equations:
\begin{equation}\label{eq:charac}
    a_{j,j}(\bkappa^*,\omega)  - \left\langle \hat\cA_j^{-1}\hat\ba_j(\bkappa^*,\omega),\hat\ba_j(\bkappa^*,\omega) \right\rangle = 0, \quad j = 0, 1, -1.
\end{equation}

In summary, we have the following proposition for the characteristic values of \eqref{eq:subsys1}-\eqref{eq:subsys3}.

\begin{proposition}\label{prop:charac}
$\omega$ is a characteristic value of the system \eqref{eq:subsys_compact} if and only if $\omega$ is a root of the characteristic equation \eqref{eq:charac}. In addition, the dimension of the solution space for each system in \eqref{eq:subsys_compact} is 1.
\end{proposition}

\subsection{Asymptotic expansion of the eigenvalues and eigenfunctions for $\bkappa=\bkappa^*$}
In view of Propositions \ref{prop_repeated_eig} and \ref{prop:charac}, let us solve the characteristic equation \eqref{eq:charac} in $\Omega_\eps(\bkappa^*)$ for $j=0, 1$ to obtain the eigenvalues. When $j=0$, \eqref{eq:charac} reads
\begin{equation}\label{eq:chara_0}
    a_{0,0}(\bkappa^*,\omega)  - \left\langle \hat\cA_0^{-1}\hat\ba_0(\bkappa^*,\omega),\hat\ba_0(\bkappa^*,\omega) \right\rangle = 0.
\end{equation}
From Proposition \ref{prop:decomp_A}, we have
\begin{align*}
 a_{0,0}(\bkappa^*,\omega) &= \frac{1}{|Y|}\frac{1}{\omega^2-|\bkappa^*|^2}\left(3- \frac{4\pi\eps^2}{3}\left(\frac{2\pi}{a}\right)^2\right) + \beta_1 + \frac{1}{2\pi}\ln\eps + O(\eps); \\
 a_{3m,0} & =  O\left(\eps\right), \quad m \neq 0.
\end{align*}
Here $\beta_1(\omega):=\dfrac{1}{2\pi}(\ln\omega+\gamma_0)+\tG(\bkappa^*,\omega;\mathbf{0})$.
Hence \eqref{eq:chara_0} attains the expansion
$$  \frac{1}{|Y|}\frac{1}{\omega^2-|\bkappa^*|^2}\left(3- \frac{4\pi\eps^2}{3}\left(\frac{2\pi}{a}\right)^2\right) + \beta_1 + \frac{1}{2\pi}\ln\eps + O(\eps) = 0,  $$
which can be written as
\begin{equation}\label{eq:chara_00}
    \omega^2-\abs{\bkappa^*}^2 = - \frac{1}{\abs{Y}}\left(3- \frac{4\pi\eps^2}{3}\left(\frac{2\pi}{a}\right)^2\right) \left(\frac{1}{2\pi}\ln\eps + \beta_1 + O(\eps) \right)^{-1}.
\end{equation}

Similarly, when $j=1$, it follows from  Proposition \ref{prop:decomp_A} that
\begin{align*}
 a_{1,1}(\bkappa,\omega)
&= -\frac{1}{2}+\frac{\alpha\eps^2}{\omega^2-|\bkappa^*|^2}+O(\eps). \\
a_{3m+1,1} &=  O(\eps),  \quad m\neq0.
\end{align*}
Here $\alpha:=\frac{2\pi}{3|Y|}\left(\frac{2\pi}{a}\right)^2$. 
Therefore, the characteristic equation
$$a_{1,1}(\bkappa^*,\omega)  - \left\langle \hat\cA_1^{-1}\hat\ba_1(\bkappa^*,\omega),\hat\ba_1(\bkappa^*,\omega) \right\rangle = 0$$
attains the expansion
\begin{equation*}
    -\frac{1}{2}+\frac{\alpha\eps^2}{\omega^2-|\bkappa^*|^2}+O(\eps) = 0,
\end{equation*}
or equivalently,
\begin{equation}\label{eq:chara_1}
   \omega^2-\abs{\bkappa^*}^2 = \frac{1}{1+O(\eps)} \cdot 2\alpha \cdot \eps^2.
\end{equation}
It follows from  Proposition \ref{prop_repeated_eig} that $\omega$ satisfying \eqref{eq:chara_1} is also a characteristic value of \eqref{eq:charac} for $j=-1$.

Note that the $\omega$ values satisfying \eqref{eq:chara_00} and \eqref{eq:chara_1} lie in the region $\Omega_\eps(\bkappa^*)$. We arrive at the following theorem for the eigenvalues in $\Omega_\eps(\bkappa^*)$ and the corresponding eigenfunctions for $\bkappa=\bkappa^*$.

\begin{theorem}\label{thm:eigenvalue0}
If $\bkappa=\bkappa^*$, the Dirichlet problem \eqref{eq:eigen_prob}-\eqref{eq:Dirichlet} attains two eigenvalues in $\Omega_\eps(\bkappa^*)$ for $\eps\ll1$:
\begin{align*}
     \omega_1^* & = \abs{\bkappa^*} +  \frac{\alpha}{\abs{\bkappa^*}} \cdot \eps^2 + O\left(\eps^3\right), \\
      \omega_1^{**} & = \abs{\bkappa^*} -  \frac{3\pi}{\abs{Y}\abs{\bkappa^*}} \cdot \frac{1}{\ln\eps}  + O\left(\frac{1}{\ln^2\eps}\right).
\end{align*}
The corresponding eigenspaces are given by 
$$ V^*:=\mbox{span}\left\{ \cS_\eps^{\bkappa,\omega} \varphi_1, \cS_\eps^{\bkappa,\omega}\varphi_{-1}\right\}
\quad \mbox{and} \quad
V^{**}:=\mbox{span}\left\{ \cS_\eps^{\bkappa,\omega} \varphi_0\right\},
$$
where $\cS_\eps^{\bkappa,\omega}$ is the single-layer potential, and $\varphi_j\in H_{\bkappa^*,j}^{-1/2}([0,2\pi])$ with $\varphi_j(t) = e^{ijt} + O(\eps)$ in the parameter space for $j=0,\pm1$. 
\end{theorem}
\begin{proof}
 The existence of roots for the characteristic equations \eqref{eq:chara_00} and \eqref{eq:chara_1} in the region $\Omega_\eps(\bkappa^*)$ follows directly from Rouche theorem, and
the asymptotic expansions of roots $\omega_1^{**}$ and $\omega_1^*$ are obtained from the expansions of \eqref{eq:chara_00} and \eqref{eq:chara_1}.  The expansions of the eigenfunctions $\varphi_0$, $\varphi_{\pm 1}$ for the integral equation \eqref{eq:int_eq2} in the parameter space are obtained by \eqref{eq:bcj}-\eqref{eq:cj}, where we use Theorem \ref{thm:A_inv}.
Hence we obtain the eigenvalues and eigenspaces for the Dirichlet problem \eqref{eq:eigen_prob}-\eqref{eq:Dirichlet}.
\end{proof}

\begin{remark}
$\omega_1^*$ is an eigenvalue of multiplicity two. The accuracy of its asymptotic formula is demonstrated in Table \ref{tab:omega_1}. As to be shown in the next section, the dispersion surfaces $(\bkappa,\omega)$  near $(\bkappa^*,\omega_1^*)$ possess conical singularity. Thus the pair $(\bkappa^*,\omega_1^*)$ is a Dirac point, which is formed by the crossing of the first two band surfaces. $\omega_1^{**}$ is an eigenvalue of multiplicity one that is located on the third band.
\end{remark}

\begin{table}[!htbp]
    \centering
    \begin{tabular}{|| c | c c c c||}
    \hline
 $\eps$ & 1/40 & 1/20 & 1/10 & 1/5 \\  
 \hline\hline
 $\omega_1^* \cdot \frac{a}{2\pi}$ & 0.66896  & 0.67559 & 0.70172 & 0.81715 \\ 
 \hline
 $\omega_{1,0}^* \cdot \frac{a}{2\pi}$ &  0.66893 & 0.67573 & 0.70294 &  0.81177 \\
 \hline
 error & 3e-5 & 1.4e-4 & 1.2e-3 & 5.4e-3 \\
 \hline
    \end{tabular}
    \caption{The accuracy of the asymptotic formula for $\omega_1^*$ for various obstacle sizes. $\omega_1^*$ is obtained numerically, and $\omega_{1,0}^*:=\abs{\bkappa^*} +  \frac{\alpha}{\abs{\bkappa^*}} \cdot \eps^2$ represents the leading orders of the asymptotic expansion. The lattice constant is set as $a=1$. }
    \label{tab:omega_1}
\end{table}

\section{The conical singularity for the dispersion surface near the Dirac point}
 We consider the dispersion relations $(\bkappa,\omega)$  near $(\bkappa^*,\omega_1^*)$ for the eigenvalue problem \eqref{eq:eigen_prob}-\eqref{eq:Dirichlet}. In view of \eqref{eq:linear_sys2}, $(\bkappa,\omega)$ is a pair such that the system $\cA(\bkappa,\omega)\, \bc = \bzero$ attains a nontrivial solution $\bc\in\mathbb{H}^{-1/2}$. Let $\bc^*\in\mathbb{H}^{-1/2}$ be a solution to the system $\cA(\bkappa^*,\omega_1^*)\bc^*=0$.
From the discussions in Section~\ref{sec:omega1}, it is known that $\bc^* \in \mathbb{H}_*^{-1/2} := \mbox{span}\{ \bc_1, \bc_{-1}\}$, where
$\bc_j=\{c_n\}_{n\in \bbZ}\in\mathbb{H}_{\bkappa^*,j}^{-1/2}$.  We normalize the vectors such that $\norm{\bc_1}=\norm{\bc_2}=1$.

Let $\delta\bkappa=\bkappa-\bkappa^*$, $\delta\omega=\omega-\omega_1^*$, $\delta\bc = \bc - \bc^*$, and $\delta\cA(\bkappa,\omega):=\cA(\bkappa,\omega)-\cA(\bkappa^*,\omega_1^*)$. Then $\delta\bc$ satisfies 
\begin{equation}\label{eq:deltac1}
  \cA(\bkappa^*,\omega_1^*)\delta\bc= - \delta\cA(\bkappa,\omega)\bc^*-\delta\cA(\bkappa,\omega) \delta \bc.
\end{equation}
The Fredholm alternative implies that there exists a $\delta\bc$ such that \eqref{eq:deltac1} holds if and only if $\delta\cA(\bkappa,\omega)\bc^*+\delta\cA(\bkappa,\omega)\delta\bc \perp \mathbb{H}_*^{-1/2}$, or equivalently,
\begin{equation}\label{eq:solvability}
   \langle \delta\cA(\bkappa,\omega)\bc^*+\delta\cA(\bkappa,\omega) \delta \bc, \bc_j \rangle = 0, \quad j = -1, 1.
\end{equation}
In addition, when \eqref{eq:solvability} holds, there exists a unique $\delta\bc\perp\mathbb{H}_*^{-1/2}$ that solves \eqref{eq:deltac1}, and the solution of the system \eqref{eq:deltac1} takes the form
\begin{equation*}
    \delta\bc = - \cA_\perp^{-1}(\bkappa^*,\omega_1^*)\big(\delta\cA(\bkappa,\omega)\bc^*+\delta\cA(\bkappa,\omega) \delta \bc\big),
\end{equation*}
where $\cA_\perp$ denotes the restriction of $\cA$ onto the subspace that is orthogonal to $\mathbb{H}_*^{-1/2}$. Explicitly, 
\begin{equation}\label{eq:deltac2}
    \delta\bc = - \Big(\cI +\cA_\perp^{-1}(\bkappa^*,\omega_1^*)\delta\cA(\bkappa,\omega) \Big)^{-1} \cA_\perp^{-1}(\bkappa^*,\omega_1^*)\delta\cA(\bkappa,\omega)\bc^*.
\end{equation}
Combining \eqref{eq:solvability} and \eqref{eq:deltac2}, and expanding $\bc^*$ as $\bc^*=\alpha_{-1} \bc_{-1}+\alpha_1 \bc_1$ for certain $\alpha_{-1}$, $\alpha_{1} \in \mathbb{C}$, we have the following lemma for $(\bkappa, \omega)$.
\begin{lemma}\label{lem:pertdisp}
For a fixed $\eps$ that is sufficiently small, a pair of $(\bkappa,\omega)$ in the neighborhood of $(\bkappa^*,\omega_1^*)$ is on the dispersion surface if and only if  the following $2\times2$ system holds:
\begin{equation}\label{eq:pertdisp}
   \alpha_{-1}\,\big\langle (\cI + \cT) \delta\cA(\bkappa,\omega)\bc_{-1},\bc_j\big\rangle+   \alpha_{1} \, \big\langle (\cI + \cT) \delta\cA(\bkappa,\omega)\bc_{1},\bc_j\big\rangle=0,\quad j = -1, 1.
\end{equation}
where $\alpha_1$ and $\alpha_{-1}$ are constants, $\alpha_1\cdot\alpha_{-1}\neq0$ and the operator
\begin{equation*}
  \cT(\bkappa,\omega) := - \delta\cA(\bkappa,\omega)\Big(\cI +\cA_\perp^{-1}(\bkappa^*,\omega_1^*)\delta\cA(\bkappa,\omega) \Big)^{-1} \cA_\perp^{-1}(\bkappa^*,\omega_1^*).    
\end{equation*}
\end{lemma}

Let $\delta a_{m,n}:=a_{m,n}(\bkappa,\omega)-a_{m,n}(\bkappa^*,\omega_1^*)$. We have the following lemma for $\delta a_{m,n}$ by using the decomposition \eqref{eq:GH}-\eqref{eq:tG} and
a standard perturbation argument.
\begin{lemma}\label{lem:leadingGrad}
For fixed $\eps\ll1$, there holds
\begin{equation}\label{eq:leadgrad}
\begin{aligned}
    \delta a_{m,n}
     = &(\delta\bkappa,\delta\omega)\cdot\int_0^{2\pi}\int_0^{2\pi}\overline{\phi_m(t)}\left(\nabla_{\bkappa,\omega}
\GL\big(\bkappa^*,\omega_1^*;\eps(\br(t)-\br(\tau))\big)+O(1)\right)\phi_n(\tau) \, d\tau dt\\
& +O(|\delta\bkappa|^2)+O(|\delta\omega|^2)
\end{aligned}
\end{equation}
for sufficiently small $\delta\bkappa$ and $\delta\omega$, where $\GL$ is defined in \eqref{eq:GL}.
\end{lemma}

\begin{theorem}\label{thm:slope}
For fixed $\eps\ll1$, there are two distinct branches of eigenvalues $\omega_1^\pm(\bkappa)$  near the Dirac point $(\bkappa^*,\omega_1^*)$ given by
 \begin{align}
  \omega_1^+(\bkappa) &= \omega_1^*+\frac{1}{3\omega_1^*}\left(\frac{2\pi}{a}\right)|\bkappa-\bkappa^*|\big(1+O(\eps)\big)+O(|\bkappa-\bkappa^*|^2), \label{eq:dis_relation1} \\
  \omega_1^-(\bkappa) &= \omega_1^*-\frac{1}{3\omega_1^*}\left(\frac{2\pi}{a}\right)|\bkappa-\bkappa^*|\big(1+O(\eps)\big)+O(|\bkappa-\bkappa^*|^2). \label{eq:dis_relation2}
\end{align}
\end{theorem}

\begin{proof}
By Lemma \ref{lem:pertdisp}, in the neighborhood of  $(\bkappa^*,\omega_1^*)$, the pair $(\bkappa,\omega)$ satisfies 
\begin{equation*}
    \mbox{det}\left(
    \begin{matrix}
    \big\langle (\cI + \cT) \delta\cA(\bkappa,\omega)\bc_{-1},\bc_{-1}\big\rangle & \big\langle (\cI + \cT) \delta\cA(\bkappa,\omega)\bc_{1},\bc_{-1}\big\rangle \\
    \\
    \big\langle (\cI + \cT) \delta\cA(\bkappa,\omega)\bc_{-1},\bc_{1}\big\rangle & \big\langle (\cI + \cT) \delta\cA(\bkappa,\omega)\bc_{1},\bc_{1}\big\rangle
    \end{matrix}\right) =0.
\end{equation*}
Recall that $\cT(\bkappa,\omega) := - \delta\cA(\bkappa,\omega)\Big(\cI +\cA_\perp^{-1}(\bkappa^*,\omega_1^*)\delta\cA(\bkappa,\omega) \Big)^{-1} \cA_\perp^{-1}(\bkappa^*,\omega_1^*)$,
using Lemma~\ref{lem:leadingGrad}, the above equation reads
\begin{equation}\label{eq:pertdisp1}
    \mbox{det}\left(
    \begin{matrix}
    \big\langle \delta\cA(\bkappa,\omega)\bc_{-1},\bc_{-1}\big\rangle & \big\langle  \delta\cA(\bkappa,\omega)\bc_{1},\bc_{-1}\big\rangle \\
    \\
    \big\langle  \delta\cA(\bkappa,\omega)\bc_{-1},\bc_{1}\big\rangle & \big\langle  \delta\cA(\bkappa,\omega)\bc_{1},\bc_{1}\big\rangle
    \end{matrix}\right) =0
\end{equation}
up to the first order in terms of $\delta\bkappa$ and $\delta\omega$ for each entry of the matrix.

Since the $n$th component of the vectors $\bc_1$ and $\bc_{-1}$ is $0$ when $\bmod(n,3)=0$,
we only need to consider the components of $\delta a_{m,n}$ for $\bmod(m,3)\neq0$ and $\bmod(n,3)\neq0$
in \eqref{eq:pertdisp1}. 
To this end, we define the set of indices relevant to \eqref{eq:pertdisp1} as
$$ E:=\{ (m,n) \in \bbZ \times \bbZ: \bmod(m,3)\neq0 \; \mbox{and} \bmod(n,3)\neq0 \}, $$
and call such $\delta a_{m,n}$ relevant to \eqref{eq:pertdisp1}.
The complement of $E$ is denoted by $E^c$.
The component $\delta a_{m,n}$ with $(m,n)\in E^c$ 
are irrelavent to \eqref{eq:pertdisp1} and we can neglect it.

To obtain the perturbation $\delta\cA(\bkappa,\omega)$, we first compute the gradient $\nabla_{\bkappa,\omega}
\GL(\bkappa^*,\omega_1^*;\eps(\br(t)-\br(\tau)))$. The partial derivatives of $g_j(\bkappa^*,\omega_1^*,\bx)$ with respect to $\bkappa=(\kappa_1,\kappa_2)$ and $\omega$ are
\begin{equation*}
\begin{aligned}
    \partial_{\kappa_1}g_j(\bkappa^*,\omega_1^*,\bx)&=\frac{ix_1e^{i(\bkappa^*+\bq_j)\cdot \bx}}{(\omega_1^*)^2-|\bkappa^*|^2} + \frac{2(\kappa^*_1+q_{j,1})e^{i(\bkappa^*+\bq_j)\cdot \bx}}{\big((\omega_1^*)^2-|\bkappa^*|^2\big)^2},\\
    \partial_{\kappa_2}g_j(\bkappa^*,\omega_1^*,\bx)&=\frac{ix_2e^{i(\bkappa^*+\bq_j)\cdot \bx}}{(\omega_1^*)^2-|\bkappa^*|^2} + \frac{2(\kappa^*_2+q_{j,2})e^{i(\bkappa^*+\bq_j)\cdot \bx}}{\big((\omega_1^*)^2-|\bkappa^*|^2\big)^2},\\
   \partial_{\omega_1^*}g_j(\bkappa^*,\omega_1^*,\bx)&=\frac{-2\omega e^{i(\bkappa^*+\bq_j)\cdot \bx}}{\big((\omega_1^*)^2-|\bkappa^*|^2\big)^2}.
\end{aligned}
\end{equation*}
Here we have used the relation $|\bkappa^*+\bq_j|^2=|\bkappa^*|^2$ for $j=1,2,3$.
By the Taylor expansion and using the fact that $(\omega_1^*)^2-|\bkappa^*|^2=O(\eps^2)$, we obtain 
\begin{eqnarray}
\partial_{\kappa_1} \GL(\bkappa^*,\omega_1^*;\eps(\bx-\by)) &=& \left(\frac{2\pi}{a}\right)\frac{2}{\big((\omega_1^*)^2-|\bkappa^*|^2\big)^2}\frac{1}{\sqrt3}
\Bigg(-i\frac{2}{3}\frac{2\pi}{a}\eps (x_2-y_2) \nonumber\\
&&- \frac{2}{3\sqrt3}\left(\frac{2\pi}{a}\right)^2\eps^2 (x_1-y_1)(x_2-y_2)\Bigg)+O\left(\frac{1}{\eps}\right),  \label{eq:partial_kappa1} \\
\partial_{\kappa_2} \GL(\bkappa^*,\omega_1^*;\eps(\bx-\by)) &=& \left(\frac{2\pi}{a}\right)\frac{2}{\big((\omega_1^*)^2-|\bkappa^*|^2\big)^2}\frac{1}{3}\Bigg(-i2\frac{2\pi}{a}\eps (x_2-y_2) \nonumber \\ &&-\frac{1}{3}\left(\frac{2\pi}{a}\right)^2\eps^2\left((x_1-y_1)^2-(x_2-y_2)^2\right)\Bigg)
+O\left(\frac{1}{\eps}\right),  \label{eq:partial_kappa2} \\
 \partial_{\omega}  \GL(\bkappa^*,\omega_1^*;\eps(\bx-\by)) &=& \frac{-2\omega_1^*}{\big((\omega_1^*)^2-|\bkappa^*|^2\big)^2}\left(3-\frac{1}{3}\left(\frac{2\pi}{a}\right)^2\eps^2|\bx-\by|^2\right)+O\left(\frac{1}{\eps}\right). \nonumber  \\ \label{eq:partial_omega}
\end{eqnarray}
Next we compute the leading-order terms in $\delta a_{m,n}$ defined in \eqref{eq:leadgrad}. The calculation is based on the relation
\begin{equation}
\begin{aligned}
\br(t)-\br(\tau)&=\big(\cos(t)-\cos(\tau),\sin(t)-\sin(\tau)\big)\\
&=\left(\frac{1}{2}(e^{it}+e^{-it}-e^{i\tau}-e^{-i\tau}),\,\frac{1}{2i}(e^{it}-e^{-it}-e^{i\tau}+e^{-i\tau})\right).
\end{aligned}
\end{equation}

First, the leading-order term $x_2-y_2$ in the kernel $\partial_{\kappa_1} G_0(\bkappa^*,\omega_1^*;\eps(\bx-\by))$ in \eqref{eq:partial_kappa1} takes the following form 
$$x_2-y_2=\frac{1}{2i}(e^{it}-e^{-it}-e^{i\tau}+e^{-i\tau})=\frac{\pi}{i}(\phi_1(t)\phi_0(\tau)-\phi_{-1}(t)\phi_0(\tau))-\phi_0(t)\phi_1(\tau)+\phi_0(t)\phi_{-1}(\tau).$$
The term contributes $\frac{\pi}{i},-\frac{\pi}{i},-\frac{\pi}{i},\frac{\pi}{i}$ to the element $\delta a_{m,n}$ in \eqref{eq:leadgrad} for $(m,n)=(1,0),(-1,0),(0,-1),(0,1)$ respectively. However,
each of these pairs of $(m,n) \notin E$ since either $m=0$ or $n=0$, and the corresponding element $\delta a_{m,n}$  will not contribute to the equation \eqref{eq:pertdisp1}.

The next leading-order term, $(x_1-y_1)(x_2-y_2)$ in $\partial_{\kappa_1} G_0(\bkappa^*,\omega_1^*;\eps(\bx-\by))$ in \eqref{eq:partial_kappa1}, takes the following form
\begin{equation*}
\begin{aligned}
    (x_1-y_1)(x_2-y_2)=&\frac{1}{2}(e^{it}+e^{-it}-e^{i\tau}-e^{-i\tau})\cdot \frac{1}{2i}(e^{it}-e^{-it}-e^{i\tau}+e^{-i\tau})\\
    =&\frac{\pi}{2i}(-\phi_1(t)\phi_1(\tau)+\phi_1(t)\phi_{-1}(\tau)-\phi_{-1}(t)\phi_1(\tau)+\phi_{1}(t)\phi_{-1}(\tau)\\
     &-\phi_1(t)\phi_1(\tau)+\phi_{-1}(t)\phi_1(\tau)-\phi_1(t)\phi_{-1}(\tau)+\phi_{-1}(t)\phi_{-1}(\tau))+\text{remainder terms}\\
    =&\frac{\pi}{i}(-\phi_1(t)\phi_1(\tau)+\phi_{-1}(t)\phi_{-1}(\tau))+\text{remainder terms}.
\end{aligned}
\end{equation*}
In the above, the remainder terms only contribute to $\delta a_{m,n}$ for $(m,n) \in E^c$ and we can neglect them.
Thus the term $(x_1-y_1)(x_2-y_2)$ contributes $-\frac{\pi}{i},\frac{\pi}{i}$ to $\delta a_{m,n}$ in \eqref{eq:leadgrad} for $(m,n)=(1,-1),(-1,1) \in E$ respectively.

Similarly, it is straightforward to calculate and verify that among the relevant elements in $\delta a_{m,n}$, the term $(x_1-y_1)^2-(x_2-y_2)^2$ in $\partial_{\kappa_2} \GL(\bkappa^*,\omega_1^*;\eps(\bx-\by))$ of \eqref{eq:partial_kappa2} contributes $-2\pi,-2\pi$ to $(1,-1),(-1,1)$ respectively, and the term $|\bx-\by|^2$ in $\partial_{\omega} \GL(\bkappa^*,\omega_1^*;\eps(\bx-\by))$ of \eqref{eq:partial_omega} contributes $-2\pi,-2\pi$ to $(1,-1),(-1,1)$ respectively.

In view of the expansions \eqref{eq:partial_kappa1}-\eqref{eq:partial_omega} and summing up their contributions to $\delta\cA$, we see that each entry in the $2\times2$ matrix of \eqref{eq:pertdisp1} takes the form
\begin{equation}
\begin{aligned}
   \big\langle \delta\cA(\bkappa,\omega)\bc_{i},\bc_{j}\big\rangle  = & \left(\frac{2\pi}{a}\right)^2\frac{\eps^2}{\big((\omega_1^*)^2-|\bkappa^*|^2\big)^2}\Bigg(\frac{-4}{9}\left(\frac{2\pi}{a}\right)\delta\kappa_1\langle A_1 \bc_i,\bc_j\rangle \\
   & -\frac{2}{9}\left(\frac{2\pi}{a}\right)\delta\kappa_2\langle A_2
   \bc_i,\bc_j\rangle 
   + \delta\omega\frac{2\omega_1^*}{3}\langle A_3
   \bc_i,\bc_j\rangle\Bigg) +O\left(\frac{1}{\eps}\max(|\delta\bkappa|,|\delta\omega|)\right),
\end{aligned}    
\end{equation}
where 
$$
A_1=\left(
   \begin{matrix}
   \ddots &&&&\\
          &0 &*&\frac{\pi}{i}&\\
          &* &**&*&\\
          &-\frac{\pi}{i} &* &0&\\
          &&&&\ddots\\
   \end{matrix}\right), A_2=\left(
   \begin{matrix}
   \ddots &&&&\\
          &0 &*&-2\pi&\\
          &* &**&*&\\
          &-2\pi &* &0&\\
          &&&&\ddots\\
   \end{matrix}\right), 
   A_3=\left(
   \begin{matrix}
   \ddots &&&&\\
          &-2\pi &*&0&\\
          &* &**&*&\\
          &0 &* &-2\pi&\\
          &&&&\ddots\\
   \end{matrix}\right).
$$
The element with a double star corresponds to the $(0,0)$ entry, and the elements that are labeled as $*$, $**$ or are not displayed are either  $0$ or irrelevant to \eqref{eq:pertdisp1}.
Therefore, up to the leading orders in $\delta\bkappa$ and $\delta\omega$, the relation between $\delta\bkappa$ and $\delta\omega$ is given by
\begin{equation*}
\det\left(
   \begin{matrix}
   \medskip
          \big(-\frac{4}{3}\omega_1^*+O(\eps)\big)\delta\omega& \frac{2\pi}{a}\Big(\big(-\frac{4}{9i}+O(\eps)\big)\delta\kappa_1+\big(\frac{4}{9}+O(\eps)\big)\delta\kappa_2\Big)\\
          \frac{2\pi}{a}\Big(\big(\frac{4}{9i}+O(\eps)\big)\delta\kappa_1+\big(\frac{4}{9}+O(\eps)\big)\delta\kappa_2 \Big) &\big(-\frac{4}{3}\omega_1^*+O(\eps)\big)\delta\omega
   \end{matrix}\right)=0.
\end{equation*}
Hence the dispersion relations \eqref{eq:dis_relation1}-\eqref{eq:dis_relation2} follow and the proof is complete.
\end{proof}

\section{Dirac points for higher frequency bands of the Dirichlet problem}
For $\bkappa=\bkappa^*$, the eigenvalues located at higher bands lie in the neighborhood of the singular frequency $\bar\omega_n(\bkappa^*)\in\Omega_{\text{sing}}(\bkappa^*)$ for $n=2, 3, \cdots$. It can be calculated explicitly that
\begin{equation*}
   \bar\omega_2(\bkappa^*) = 2|\bkappa^*|, \,\,\,
   \bar\omega_3(\bkappa^*) = \sqrt{7} \, |\bkappa^*|,\,\,\,
   \bar\omega_4(\bkappa^*) = \sqrt{13} \, |\bkappa^*|,\,\,\,
   \bar\omega_5(\bkappa^*) = 4|\bkappa^*|, \cdots
\end{equation*}
As shown below, the studies of Dirac points near $\bar\omega_n(\bkappa^*)$ are parallel to the Dirac point at the crossing of the first two band surfaces studied in Sections 3 and 4. The difference only lies in the choice of the set of vectors
 $   \Lambda^*_0(\bar\omega_n):=\left\{\bq\in\Lambda^*: |\bkappa^*+\bq|=|\bar\omega_n(\bkappa^*)|\right\}$
and the decomposition of the Green function \eqref{eq:GH}-\eqref{eq:tG} associated with the set $\Lambda^*_0(\bar\omega_n)$.
The cardinal number of the set $\Lambda^*_0(\bar\omega_n)$ is a multiple of $3$. 
Here we only discuss the Dirac point near $\bar\omega_2(\bkappa^*)$ and $\bar\omega_3(\bkappa^*)$ when the cardinal number of $\Lambda^*_0(\bar\omega_2)$ and $\Lambda^*_0(\bar\omega_3)$ is $3$ and $6$ respectively. We highlight the similarities and differences in theses two scenarios. The Dirac points near $\bar\omega_n(\bkappa^*)$ for $n>3$ can be obtained similarly.

\begin{remark}
We point out that there also exist Dirac points when the Bloch wave vector $\bkappa\neq\bkappa^*$ for higher frequency bands. This is illustrated in the left panel of Figure \ref{fig:band_structure2}, where Dirac points appear at the crossing of the seventh and eighth, and the tenth and eleventh band surfaces respectively between the $M\Gamma$ segments in the Brillouin zone. The mathematical analysis of these Dirac points are beyond the scope of this paper and will be investigated in the future.
\end{remark}

\begin{figure}[!htbp]
    \centering
    \includegraphics[width=11cm]{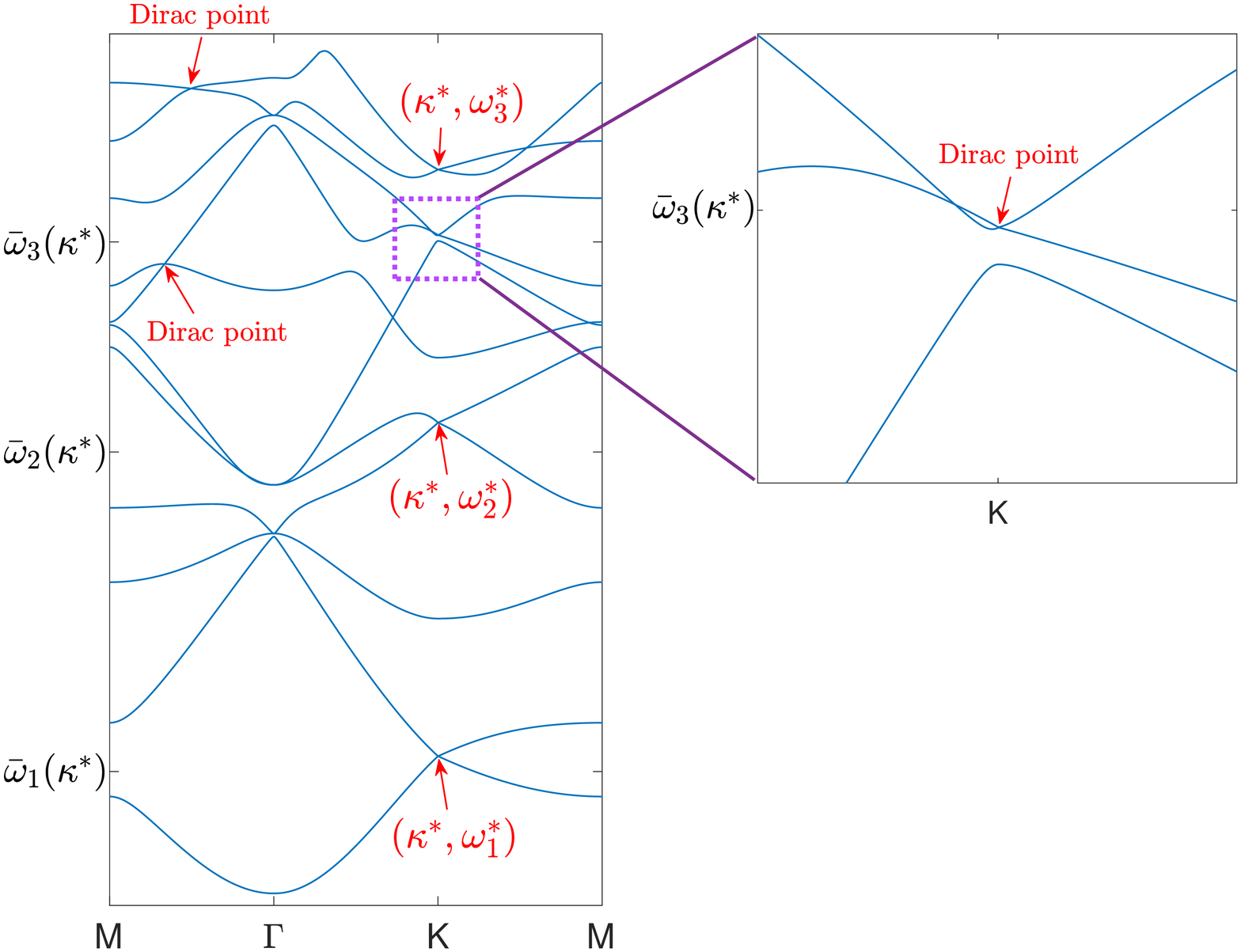}
    \vspace*{-15pt}
    \caption{Left: the dispersion curves along the segments $M\to\Gamma\to K \to M$ over the Brillouin zone when $\eps=0.1$. When $\bkappa=\bkappa^*$, Dirac points appear at the crossing of the first and second, the fourth and fifth, the tenth and eleventh band surfaces, respectively. Dirac points also appear at the crossing of the seventh and eighth, and the tenth and eleventh band surfaces between the $M\Gamma$ segments.
    Right: Zoomed view of the crossing of the eighth and ninth band surfaces. One Dirac point is formed by these two band surfaces in the $O(\eps^4)$ neighborhood of the singular frequency $\bar\omega_3(\bkappa^*)$ when $\bkappa^*=K$. }
    \label{fig:band_structure2} 
    \vspace*{-10pt}
\end{figure}

\subsection{Dirac point near $\bar\omega_2(\bkappa^*)$}
The set $\Lambda^*_0(\bar\omega_2)=\left\{\bq_1,\bq_2,\bq_3\right\}$, where
$\bq_1=\frac{2\pi}{a}(-\sqrt3,-1)^T$,
$\bq_2=\frac{2\pi}{a}(\frac{1}{\sqrt{3}},-1)^T$,  and $\bq_3=\frac{2\pi}{a}(-\frac{1}{\sqrt3},1)^T$.
To obtain the Dirac point, we solve the characteristic values of the systems \eqref{eq:subsys1}-\eqref{eq:subsys3} in the region: 
\begin{equation*}
    \Omega_\eps(2\bkappa^*): = \left\{ \omega \in \mathbb R^+: \frac{1}{\abs{Y}} \left(\frac{2\pi}{a}\right)^2\eps^2 \le \omega^2-\abs{2\bkappa^*}^2 \le \frac{2}{\abs{Y}} \left(\frac{2\pi}{a}\right)^2\eps^2  \right\}.
\end{equation*}
Using the set $\Lambda^*_0(\bar\omega_2)$ above, the Green function $G(\bkappa^*,\omega,\bx)$ is decomposed into the three parts as follows:
\begin{equation*}
 H_0(\omega; \bx):= -\frac{i}{4}  H_0^{(1)}(\omega|\bx|),
\end{equation*}
\begin{equation*}
 \GL(\bkappa, \omega;\bx):=\frac{1}{|Y|}\sum_{\bq\in \Lambda_0^*(\bar\omega_2)} \frac{e^{i(\bkappa+\bq)\cdot\bx}}{\omega^2-|\bkappa+\bq|^2},
\end{equation*}
and 
\begin{equation*}
 \tG(\bkappa, \omega; \bx):=G(\bkappa,\omega;\bx)-H_0(\omega;\bx)-\GL(\bkappa,\omega;\bx).
\end{equation*}
Similar to \eqref{eq:GL}, $\GL(\bkappa, \omega;\bx)$ above also consists of three modes and it extracts the singularity of the Green's function when $\omega$ approaches the singular frequency $\bar\omega_2(\bkappa^*)$. In addition, it attains the Taylor expansion 
\begin{equation*}
     \GL(\bkappa^*,\omega;\bx)=\frac{1}{|Y|}\frac{1}{\omega^2-|2\bkappa^*|^2} \left( 3 - \frac{4}{3}\left(\frac{2\pi}{a}\right)^2 |\bx|^2\right)+\GL^\infty(\bkappa^*,\omega;\bx).
\end{equation*}
Therefore, by repeating calculations in Section 3, it can be shown that for $\omega\in \Omega_{2\eps}(\bkappa^*)$, 
the matrix $\cA$ can be decomposed as $\cA=\cD+ \eps \, \cE$, where
$\cD:=\diag(d_n)_{n\in\mathbb Z}$ with
\begin{equation*}
   d_n=
   \begin{cases}
  \dfrac{1}{2\pi}(\ln\eps+\ln\omega+\gamma_0))+\dfrac{1}{|Y|}\dfrac{1}{\omega^2-|2\bkappa^*|^2}\left(3- \dfrac{16\pi\eps^2}{3}\left(\dfrac{2\pi}{a}\right)^2\right)+\tG(\bkappa^*,\omega;\mathbf{0}) & n=0, \\
    -\dfrac{1}{2}+\dfrac{1}{|Y|}\dfrac{1}{\omega^2-|2\bkappa^*|^2}\dfrac{8\pi\eps^2}{3}\left(\dfrac{2\pi}{a}\right)^2, & n=\pm1, \\
   -\dfrac{1}{2|n|}, & |n|>1,
   \end{cases}
\end{equation*}
and $\cE=[e_{m,n}]$ is bounded from $\mathbb{H}^{-1/2}$ to $\mathbb{H}^{1/2}$.
As such the subsystems \eqref{eq:subsys2}-\eqref{eq:subsys3} reduce again to the characteristic equation \eqref{eq:charac}, and we can obtain the Dirac point, which is formed by the crossing of the fourth and fifth band surfaces shown in Figure \ref{fig:band_structure2} (left).

\begin{theorem}\label{thm:eigenvalue2}
Let $\bkappa=\bkappa^*$ and $\eps\ll1$,  \eqref{eq:eigen_prob}-\eqref{eq:Dirichlet} attains an eigenvalue  
\begin{equation*}
     \omega_2^*  = \abs{2\bkappa^*} +  \frac{2\alpha}{\abs{\bkappa^*}} \cdot \eps^2 + O\left(\eps^3\right)
\end{equation*}
in $\Omega_{\eps}(2\bkappa^*)$, with the corresponding eigenspace
$$ V_2^*:=\mbox{span}\left\{ \cS_\eps^{\bkappa,\omega} \varphi_1,  \cS_\eps^{\bkappa,\omega} \varphi_{-1} \right\}, \mbox{wherein} \; \varphi_1(t) = e^{it} + O(\eps) \; \mbox{and} \;  \varphi_{-1}(t) = e^{-it} + O(\eps).  $$
In the neighborhood of $(\bkappa^*,\omega_2^*)$, there are two distinct branches of eigenvalues given by
\begin{equation*}
     \omega_2^\pm(\bkappa) =  \omega_2^* \pm \frac{4}{3\omega_2^*}\left(\frac{2\pi}{a}\right)|\bkappa-\bkappa^*|\big(1+O(\eps)\big)+O(|\bkappa-\bkappa^*|^2). 
\end{equation*}
\end{theorem}

\subsection{Dirac point near $\bar\omega_3(\bkappa^*)$}
Different from the studies of Dirac points near $\bar\omega_1(\bkappa^*)$ and $\bar\omega_2(\bkappa^*)$, here the set $\Lambda^*_0(\bar\omega_3)$ consists of 6 vectors:
\begin{eqnarray*}
&& \bq_1=\frac{2\pi}{a}\left(-\frac{2}{\sqrt3},-2\right)^T, \; \bq_2=\frac{2\pi}{a}(-\sqrt3,1)^T, \; \bq_3=\frac{2\pi}{a}(0, -2)^T, \\
&& \bq_4=\frac{2\pi}{a}\left(\frac{1}{\sqrt3},1\right)^T, \; \bq_5=\frac{2\pi}{a}\left(-\frac{4}{\sqrt3},0\right)^T, \; \bq_6=\frac{2\pi}{a}\left(\frac{2}{\sqrt3},0\right)^T.
\end{eqnarray*}
Correspondingly, we define
\begin{equation*}
     \GL(\bkappa, \omega;\bx):=\frac{1}{|Y|}\sum_{\bq\in \Lambda_0^*(\bar\omega_3)} \frac{e^{i(\bkappa+\bq)\cdot\bx}}{\omega^2-|\bkappa+\bq|^2},
\end{equation*}
which attains the Taylor expansion 
\begin{equation*}
     \GL(\bkappa^*,\omega;\bx)=\frac{1}{|Y|}\frac{1}{\omega^2-7|\bkappa^*|^2} \left( 6 - \frac{14}{3}\left(\frac{2\pi}{a}\right)^2 |\bx|^2\right)+\GL^\infty(\bkappa^*,\omega;\bx).
\end{equation*}

First, repeating the procedure in Section 4, it can be shown that in the $O(\eps^2)$ neighborhood of the singular frequency $\bar\omega_3(\bkappa^*)$,
there exits a Dirac point given by
\begin{equation*}
     \omega_3^*  = \sqrt{7}\abs{\bkappa^*} +  \frac{2\sqrt{7}\alpha}{\abs{\bkappa^*}} \cdot \eps^2 + O\left(\eps^3\right).
\end{equation*}
The two branches of eigenvalues near $(\bkappa^*,\omega_3^*)$ are
\begin{equation*}
  \omega_3^\pm(\bkappa) = \omega_3^*\pm\frac{20}{21\omega_3^*}\left(\frac{2\pi}{a}\right)|\bkappa-\bkappa^*|\big(1+O(\eps)\big)+O(|\bkappa-\bkappa^*|^2).
\end{equation*}
This Dirac point is formed by crossing of the tenth and eleventh band surfaces of the eigenvalue problem (see Figure \ref{fig:band_structure2}, left). 
In this scenario, there also exist a Dirac point 
in the $O(\eps^4)$ neighborhood of the singular frequency $\bar\omega_3(\bkappa^*)$ for subsystems \eqref{eq:subsys2} and \eqref{eq:subsys3} when $\bkappa=\bkappa^*$. This is shown in Figure \ref{fig:band_structure2} (right).  One can use the elements $a_{1,1}$, $a_{-2,1}$, $a_{1,-2}$, and $a_{-2,-2}$ in the matrix $\cA$ to set up the characteristic equation and obtain the leading-order term of the eigenvalue. Here we omit the very technical calculations for this Dirac point.

\section{Dirac points for the Neumann eigenvalue problem}
In this section, we consider the Dirac points for the Neumann eigenvalue problem \eqref{eq:eigen_prob}\eqref{eq:Neumann}. The asymptotic expansions of the Dirac points when $\bkappa=\bkappa^*$ are obtained by using the double-layer integral operator to derive  the characteristic equation. The key steps are parallel to the Dirichlet eigenvalue problem. Therefore, in what follows we describe the main procedure briefly and highlight the main difference from the Dirichlet problem.

First, applying the Green's identity, the Bloch mode $\psi$ for the eigenvalue problem \eqref{eq:eigen_prob}\eqref{eq:Neumann} can be represented by the double-layer potential
\begin{equation}\label{eq:double_layer} 
\psi(\bx)=-\int_{\by\in \partial D_\eps} \frac{\partial G(\bkappa,\omega;\bx-\by)}{\partial \nu(\by)} \psi(\by) \, ds_\by \quad \mbox{for} \; \bx \in   Y_\eps.
\end{equation}
Taking the limit when $\bx$ approaches the obstacle boundary and using the jump relation for the double-layer potential, we obtain the integral equation over $\partial D_\eps$:
\begin{equation*}
\frac{\psi(\bx)}{2} + \int_{\by\in \partial D_\eps} \frac{\partial G(\bkappa,\omega;\bx-\by)}{\partial \nu(\by)} \psi(\by) \, ds_\by = 0 \quad \mbox{for} \; \bx \in   \partial D_\eps.
\end{equation*}
By the change of variables, the above integral equation reads
\begin{equation}\label{eq:BIE_db}
    \left(\frac{1}{2} \cI + \eps \, \cK_\eps^{\bkappa,\omega} \right) \psi = 0.
\end{equation}
Here the double-layer integral operator $\cK_\eps^{\bkappa,\omega}$ is defined over $\partial D_1$ and takes the form
\begin{equation*}
    [\cK_\eps^{\bkappa,\omega}\psi](\bx):= \int_{\by\in \partial D_1} \Theta(\bkappa,\omega;\eps(\bx-\by)) \psi(\by) \, ds_\by, \quad \bx\in\partial D_1,
\end{equation*}
in which the kernel
\begin{equation*}
\Theta(\bkappa, \omega; \bx-\by):=\dfrac{\partial G(\bkappa,\omega;\bx-\by)}{\partial \nu(\by)}=\frac{1}{|Y|}\sum_{\bq\in \Lambda^*} -\frac{i (\bkappa+\bq) \cdot \nu(\by) \, e^{i(\bkappa+\bq)\cdot(\bx-\by)}}{\omega^2-|\bkappa+\bq|^2} \quad \mbox{for} \; \bx\neq\by.
\end{equation*}
The operator $\cK_\eps^{\bkappa,\omega}$ is bounded on $H^{1/2}(\partial D_1)$. 

To solve for the eigenpair $(\omega,\psi)$ of the integral equation \eqref{eq:BIE_db}, we expand $\psi(\br(t)) \in H^{1/2}([0,2\pi])$ as 
$\displaystyle{\psi(\br(t))=\sum_{n=-\infty}^\infty c_n \phi_n(t)}$ and define the matrix $\cA=[a_{m,n}]$, wherein
\begin{equation}\label{eq:amn_Neumann}
 a_{m,n}(\bkappa,\omega)= \frac{1}{2}\delta_{n,n} + \eps\, \left( \phi_m, \cK_\eps^{\bkappa,\omega}\phi_n \right). 
\end{equation}
Then the eigenvalue problem reduces to solving nonzero vector $\bc=\{c_n\}_{n\in\bbZ} \in \mathbb{H}^{1/2}$ for 
the linear system $\cA(\bkappa,\omega)\, \bc = \bzero$. The key steps for deriving the eigenvalues when $\bkappa=\bkappa^*$ are summarized as follows:

\begin{itemize}
    \item [(i)] \textit{Decomposition of the system:} First, the decomposition of the whole linear system $\cA(\bkappa,\omega)\, \bc = \bzero$ into the three subsystems in the form of \eqref{eq:subsys1}-\eqref{eq:subsys3} is based upon the important fact that $a_{m,n}\neq0$ only if $\bmod(m-n,3)=0$. This also holds true for the coefficients $a_{m,n}$ defined in \eqref{eq:amn_Neumann} because of the relation
    \begin{equation*}
    u_n(\bx) := \cK_\eps^{\bkappa^*,\omega}[\phi_n]
     = e^{\frac{i2n\pi}{3}} u_n(R\bx),
    \end{equation*}
    which follows from $\nabla G(\bkappa^*,\omega;\bx) = R^T \nabla G(\bkappa^*,\omega;R\bx)$ for $\bx\neq 0$ and $\nu(\by) =\by$ for $\by \in \partial D_1$. The proof can be carried out by repeating the lines in Lemma \ref{lem:amn}.
    \item[(ii)] \textit{Decomposition of the matrix $\cA$:}  The decomposition of $\cA$ follows from the decomposition of the
    integral operator $\cK_\eps^{\bkappa^*,\omega}$. 
    For $\omega$ near the singular frequency $\bar\omega_1(\bkappa^*)$, we decompose the kernel $\Theta(\bkappa, \omega; \bx-\by)$ into three parts as
    \begin{eqnarray*}
         \Theta(\bkappa^*, \omega; \bx-\by) &=& \Theta_{H_0}(\bx-\by) + \Theta_{\Lambda_0^*}( \bx-\by) + \tilde \Theta(\bx-\by) \\
         &=&
    \dfrac{\partial H_0(\omega;\bx-\by)}{\partial \nu(\by)} + \dfrac{\partial \GL(\bkappa^*,\omega;\bx-\by)}{\partial \nu(\by)} + \dfrac{\partial \tG(\bkappa^*,\omega;\bx-\by)}{\partial \nu(\by)},
    \end{eqnarray*}
    where $H_0$, $\GL$ and $\tG$ are given in \eqref{eq:GH}-\eqref{eq:tG}. Correspondingly, $\cK_\eps^{\bkappa^*,\omega}$ is decomposed as
    $\cK_\eps^{\bkappa^*,\omega}= \KH + \KL + \tK$,
    in which $\KH$, $\KL$ and $\tK$ are the integral operators with the kernel $\Theta_{H_0}(\eps(\bx-\by))$, $\Theta_{\Lambda_0^*}(\eps(\bx-\by))$, and $\tilde \Theta(\eps(\bx-\by))$ respectively. In addition, $\KH$ and $\tK$ are bounded operators on $H^{1/2}(\partial D_1)$ for $\omega$ near $\bar\omega_1(\bkappa^*)$. 
    
    A key observation is that $\Theta_{\Lambda_0^*}$ attains the expansion
    \begin{equation}
     \Theta_{\Lambda_0^*}(\eps (\bx-\by))=\frac{1}{|Y|}\frac{1}{\omega^2-|\bkappa^*|^2} \left( - \frac{\eps}{3}\left(\frac{2\pi}{a}\right)^2 |\bx-\by|^2 + O(\eps^2)\right),
\end{equation}
    which yields
    \begin{equation*}
   \left( \phi_m, \KL\phi_n \right) =
   \begin{cases}
   \dfrac{1}{|Y|}\dfrac{1}{\omega^2-|\bkappa^*|^2}\left(- \dfrac{4\pi\eps}{3}\left(\dfrac{2\pi}{a}\right)^2+O(\eps^2)\right) \quad &m=n=0, \\
   \dfrac{1}{|Y|}\dfrac{1}{\omega^2-|\bkappa^*|^2}
  \left(\dfrac{2\pi\eps}{3}\left(\dfrac{2\pi}{a}\right)^2+O(\eps^2)
  \right) \quad &m=n=\pm1\\
   \dfrac{1}{|Y|}\dfrac{1}{\omega^2-|\bkappa^*|^2}\, O\left(\eps^{\max(3,|m|,|n|)}\right) \quad &\text{otherwise}.
   \end{cases},
\end{equation*}
Therefore, recall that 
\begin{equation*}
     a_{m,n}(\bkappa,\omega)= \frac{1}{2}\delta_{n,n} + \eps\, \left( \phi_m, \KL\phi_n \right) + \eps\, \left( \phi_m, \KH\phi_n \right) + \eps\, \left( \phi_m, \tK\phi_n \right),
\end{equation*}
the matrix $\cA$ can be decomposed as $\cA=\cD+ \eps \, \cE$ near the singular frequency $\bar\omega_1(\bkappa^*)$, where
$\cD:=\diag(d_n)_{n\in\mathbb Z}$ with
\begin{equation*}
   d_n=
   \begin{cases}
   \dfrac{1}{2} +
  \dfrac{1}{|Y|}\dfrac{1}{\omega^2-|\bkappa^*|^2}\left(- \dfrac{4\pi\eps^2}{3}\left(\dfrac{2\pi}{a}\right)^2\right), & n=0, \\
    \dfrac{1}{2}+\dfrac{1}{|Y|}\dfrac{1}{\omega^2-|\bkappa^*|^2}\dfrac{2\pi\eps^2}{3}\left(\dfrac{2\pi}{a}\right)^2, & n=\pm1, \\
   \dfrac{1}{2}, & |n|>1,
   \end{cases}
\end{equation*}
and $\cE=[e_{m,n}]$ is bounded on $\mathbb{H}^{1/2}$.

    \item[(iii)] \textit{Characteristic equation:} We rewrite 
   each system in \eqref{eq:subsys1}-\eqref{eq:subsys3} as
\begin{equation}\label{eq:subsysem_compact2}
a_{j,j} \, c_j + \left\langle \hat\bc_j,\hat\ba_j \right\rangle \,  = 0, \quad  \hat\cA_j \hat\bc_j + c_j \hat\ba_j = \bzero, \quad j = 0, \pm 1,
\end{equation}
where the vectors $\hat\ba_j$, $\hat\bc_j$ and $\hat\cA_j$ are defined by \eqref{eq:hat_a}-\eqref{eq:hat_Aj}. Using the decomposition of the matrix $\cA$ in (ii), then $\hat\cA_j=\hat\cD_j+ \eps \, \hat\cE_j$ $\left(\mathbb{H}_{\bkappa^*,j}^{1/2} \to \mathbb{H}_{\bkappa^*,j}^{1/2} \right)$ is invertible. As such the system \eqref{eq:subsysem_compact2} reduces to the characteristic equation: 
\begin{equation*}
    a_{j,j}(\bkappa^*,\omega)  - \left\langle \hat\cA_j^{-1}\hat\ba_j(\bkappa^*,\omega),\hat\ba_j(\bkappa^*,\omega) \right\rangle = 0, \quad j = 0, 1, -1.
\end{equation*}
Explicitly, we obtain
\begin{equation}\label{eq:char_Neumann}
 \frac{1}{2}-\frac{2\alpha\eps^2}{\omega^2-|\bkappa^*|^2}+O(\eps) = 0
 \quad \mbox{and} \quad
 \frac{1}{2}+\frac{\alpha\eps^2}{\omega^2-|\bkappa^*|^2}+O(\eps) = 0
\end{equation}
for $j=0$ and $j=\pm1$ respectively, where $\alpha:=\frac{2\pi}{3|Y|}\left(\frac{2\pi}{a}\right)^2$. 

\end{itemize}

Solving the second equation in \eqref{eq:char_Neumann} for $j=\pm1$, we obtain the Dirac point near $\bar\omega_1(\bkappa^*):=|\bkappa^*|$:
\begin{theorem}
If $\bkappa=\bkappa^*$, the integral equation \eqref{eq:BIE_db} attains the characteristic value
$$ \omega_1^*  = \abs{\bkappa^*} -  \frac{\alpha}{\abs{\bkappa^*}} \cdot \eps^2 + O\left(\eps^3\right)$$ in the neighborhood of $\bar\omega_1(\bkappa^*)$.
The corresponding eigenspace $V^*:=\mbox{span}\left\{\psi_1(t),\psi_{-1}(t)\right\}$,
where $\psi_j(t) = e^{ijt} + O(\eps) \in H_{\bkappa^*,j}^{1/2}([0,2\pi])$ for $j=\pm1$.
\end{theorem}
The Dirac point $(\bkappa^*, \omega_1^*)$ is formed by the crossing of the first two band surfaces shown in Figure \ref{fig:band_structure3}. Similar to the Dirichlet problem, the Neumann eigenvalue $\omega_1^*$ is also located in the $O(\eps^2)$ neighborhood of the singular frequency $\bar\omega_1(\bkappa^*)$, but it is smaller than $\bar\omega_1(\bkappa^*)$. One major difference of the two eigenvalue problems is that, for small $\eps$, the Dirichlet Dirac point $(\bkappa^*, \omega_1^*)$ is the extreme point  of the first two band surfaces (maximum and minimum respectively) and a band gap can be opened near $(\bkappa^*, \omega_1^*)$ if the honeycomb lattice is perturbed suitably (cf. Figure \ref{fig:band_structure}). This is not the case for the Neumann Dirac point and the obstacle needs to be large enough for $(\bkappa^*, \omega_1^*)$ to be an extreme point (cf. Figure \ref{fig:band_structure3}).
Finally, one can obtain the Dirac points for the higher frequency bands by decomposing the integral operator $\cK_\eps^{\bkappa^*,\omega}$ properly and repeating the above calculations.

\begin{figure}[!htbp]
    \centering
    \includegraphics[width=14cm]{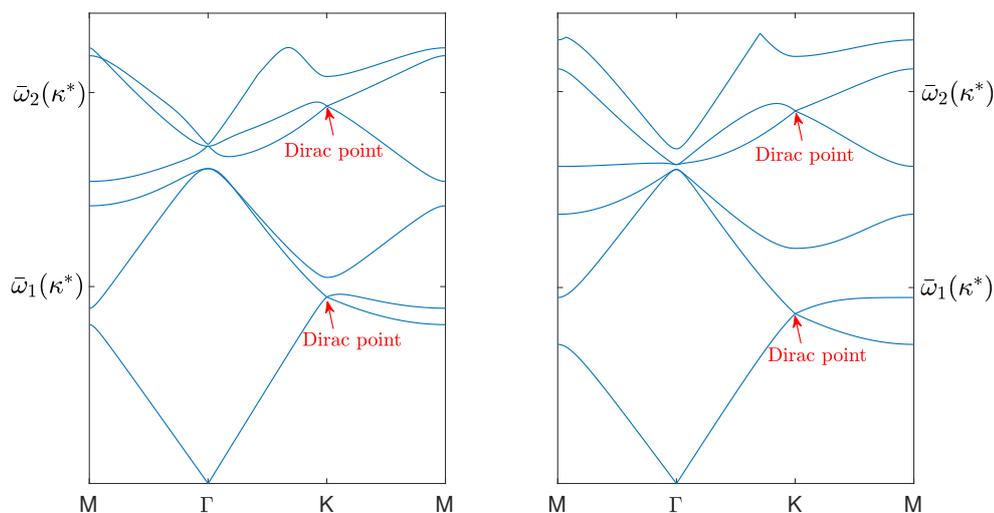}
    \caption{The dispersion curves along the segments $M\to\Gamma\to K \to M$ over the Brillouin zone when $\eps=0.1$ (left) and $\eps=0.2$ (right). The lattice constant $a=1$. }
    \label{fig:band_structure3} 
\end{figure}

\bibliography{references}

\end{document}